\documentclass{amsart}

\usepackage{amsmath,amssymb,amsthm,bbm} 
\usepackage{hyperref} 
\usepackage{geometry} 
\usepackage{xcolor} 
\usepackage{tikz} 
\usepackage{xfp} 
\usepackage{stmaryrd} 

\definecolor{theme}{RGB}{15, 87, 24} 
\definecolor{lighttheme}{RGB}{51, 153, 63} 
\definecolor{block}{RGB}{102, 60, 0} 
\definecolor{lightblock}{RGB}{255, 238, 214} 
\definecolor{alert}{RGB}{212, 43, 43} 

\def\pagespace{3cm}
\hypersetup{
    colorlinks=true,
    linkcolor=theme,
    citecolor=lighttheme
}

\newtheorem{thm}{Theorem}[section]
\newtheorem{prop}[thm]{Proposition}
\newtheorem{lemma}[thm]{Lemma}
\newtheorem{cor}[thm]{Corollary}

\newcommand{\pskip}{\left.\hspace{0.5cm}\right.} 
\newlength{\EUflag}
\newcommand{\bN}{\mathbb{N}}
\newcommand{\bZ}{\mathbb{Z}}

\newcommand{\prob}{\mathbb{P}}
\newcommand{\expec}{\mathbb{E}}
\newcommand{\Sn}{\mathcal{S}_n}

\newcommand{\inv}{\mathrm{Inv}}
\newcommand{\sign}{\mathrm{sgn}}
\newcommand{\bern}[1]{\textsc{Bernoulli}(#1)}
\newcommand{\geom}[1]{\textsc{Geometric}(#1)}
\newcommand{\pgeom}[1]{\textsc{Geometric}^+(#1)}
\newcommand{\sbgeom}[1]{\textsc{Geometric}^\star(#1)}
\newcommand{\BST}{\mathbb{T}}
\renewcommand{\root}{\varnothing}
\newcommand{\lv}{\mathbf{0}}
\newcommand{\rv}{\mathbf{1}}

\newcommand{\bst}[1]{T\!\left\langle#1\right\rangle}
\newcommand{\RIBST}{\mathbb{Y}}
\newcommand{\ribst}[1]{Y\!\left\langle#1\right\rangle}
\newcommand{\uv}{\Upsilon}

\newcommand{\localconv}{\overset{\mathrm{\ell.a.s.}}{\longrightarrow}}
\renewcommand{\mp}[1]{X_{#1,q}}
\newcommand{\imp}{\mathcal{X}_q}
\newcommand{\imt}{\mathcal{Y}_q}
\newcommand{\imroot}{o}

\newcommand{\rootconv}{\overset{\root}{\longrightarrow}}
\newcommand{\ghpconv}{\overset{\mathrm{GHP}}{\longrightarrow}}
\newcommand{\metric}{E}
\newcommand{\metricspace}{\mathcal{\metric}}
\newcommand{\ghplimit}{\metricspace_u}
\newcommand{\subtreeconv}{\overset{\mathrm{ssc}}{\longrightarrow}}

\newcommand{\leftT}{\BST_\lv}
\newcommand{\reco}[1]{R_{#1,q}}
\newcommand{\free}[1]{F_{#1,q}}
\newcommand{\disteq}{\overset{d}{=}}
\newcommand{\cond}[1]{\ensuremath{\mathrm{(\hyperlink{cond#1}{P#1})}}}
\newcommand{\asconv}{\overset{\mathrm{a.s.}}{\longrightarrow}}
\newcommand{\mt}[1]{T_{#1,q}}
\newcommand{\lastT}{T_{n,q}^L}
\newcommand{\disto}{\mathrm{dis}}
\newcommand{\dghp}{d_{\mathrm{GHP}}}
\renewcommand{\ss}{\tau}
\newcommand{\ssm}[1]{\ss_{#1,q}}

\title{Limits of Mallows trees}
\author{Beno\^it Corsini}
\date{}

\begin{document}

\maketitle

\begin{abstract}
    This article studies the limit of binary search trees drawn from Mallows permutations under various topologies.
    The main result, pertaining to the standard local topology for graphs, requires the introduction of a generalization of binary search trees to two-sided infinite sequences, referred to as \textit{redwood trees}.
    We then show that the almost-sure local limit of finite Mallows trees is the redwood tree drawn from the two-sided infinite Mallows permutation, thus corresponding to swapping the local limit and the binary search tree structure.

    Building off this result, we study various other natural topologies: the rooted topology of the local structure around the root, the Gromov-Hausdorff-Prokhorov topology of the tree seen as a metric space, and the subtree size topology of the ratio of nodes split between left and right subtrees.
    The limit of Mallows trees under these three topologies combined with the case of the local topology allow us to draw a global picture of what large Mallows trees look like from different perspective and further strengthen the relation between finite and infinite Mallows permutations.
\end{abstract}

\setcounter{tocdepth}{1}
\tableofcontents

\section{Introduction}\label{sec:intro}

Mallows permutations~\cite{mallows1957non} have recently started to attract more attention as a natural model to generalize uniform permutations.
In particular, the last decade has seen an increasing number of articles studying structural properties of Mallows permutations, such as the length of the longest monotone subsequence~\cite{basu2017limit,bhatnagar2015lengths,mueller2013length}, the number of cycles~\cite{gladkich2018cycle,he2022cycles,mukherjee2016fixed}, or the occurrence of patterns~\cite{crane2018probability,he2022central,pinsky2021permutations}.

A notable property of Mallows permutations is their strong regenerative structure~\cite{pitman2019regenerative} which allowed their definition to be extended to more complex sets than $[n]:=\{1,\ldots,n\}$, such as the one-sided infinite Mallows permutation on $\bN=\{1,2,\ldots\}$~\cite{gnedin2010q}, or the two-sided infinite Mallows permutation on $\bZ=\{\ldots,-1,0,1,\ldots\}$~\cite{gnedin2012two}.
In this work, we further highlight the relation between Mallows permutations on $[n]$, $\bN$, and $\bZ$ by studying the limit under various topologies of \textit{Mallows trees}, defined to be binary search trees drawn from Mallows permutations.

The bulk of this work focuses on the topology of the local convergence of graphs, as introduced by Benjamini and Schramm~\cite{benjamini2001recurrence}.
In that case, we show that the limit of a finite Mallows tree is almost-surely an adequately defined binary search tree on a two-sided infinite Mallows permutation (Theorem~\ref{thm:conv}).
This result is interesting to compare with the rooted local topology, where the limit of a finite Mallows tree is the binary search tree drawn from the infinite one-sided Mallows distribution (Theorem~\ref{thm:rooted}).
Finally, two other natural topologies on trees are considered: the Gromov-Hausforff-Prokhorov topology, where trees are seen as metric spaces, for which the limit is a simple metric on $[0,1]$ (Theorem~\ref{thm:ghp}), and the recently introduced subtree size topology~\cite{grubel2023note}, for which the limit is the infinite rightward path (Theorem~\ref{thm:ssc}).

\medskip

The main result of this work (Theorem~\ref{thm:conv}) provides an interesting limit swap between binary search tree structures and local limits of discrete objects.
Thus it should not fully come as a surprise:
the local structure of a finite Mallows tree seen from a uniformly chosen node only depends on the ordering of the surrounding values of a uniformly chosen entry in the permutation and behaves similarly to the tree corresponding to a two-sided infinite Mallows permutation seen from any entry.
Although swapping the two limits seems natural here, it is worth noting that a key ingredient in the proof relies on the relation between displacements in Mallows permutations and geometric random variables, and the fact that the sum of two geometric random variables is similar to the size-biased geometric distribution (see Section~\ref{sec:imt} and more precisely~\eqref{eq:geomsing} for more details on this regard).

Apart from the main results on the convergence of finite Mallows trees under various topologies, this article provides two noteworthy novelties.
First, it extends the notion of binary search tree from one-sided infinite sequences to two-sided infinite sequences, under some mild assumptions.
This extension of binary search trees is referred to as \textit{redwood trees} (see Section~\ref{sec:disclaimer} for a discussion on the name and notations for redwood trees) and is defined and studied in Section~\ref{sec:redwood}.
Second, the main result of this article proves the almost-sure local convergence of a sequence of trees, joining the rather small family of almost-sure local convergence results~\cite{dembo2010ising}.
Moreover, while proving almost-sure local convergence often relies on applying the Borel-Cantelli lemma and thus does not directly require a properly defined common space, Mallows trees are, to the best of our knowledge, the first sequence of trees with a recursive construction for which such result is known.

\subsection{Main results}\label{sec:main}

For any $n\in\bN$ and $q\in[0,\infty)$, the Mallows distribution~\cite{mallows1957non} is the distribution $\pi_{n,q}$ on the symmetric group $\Sn$ defined by
\begin{align*}
    \pi_{n,q}(\sigma)&:=\frac{q^{\inv(\sigma)}}{Z_{n,q}}\,,
\end{align*}
where $\inv(\sigma):=|\{i<j:\sigma(i)>\sigma(j)\}|$ is the number of inversions of $\sigma$ and $Z_{n,q}:=\sum_{\sigma\in\Sn}q^{\inv(\sigma)}$ is a normalizing constant.
If $\sigma=(\sigma(1),\ldots,\sigma(n))\in\Sn$ is a random permutation distributed accorded to $\pi_{n,q}$, then the permutation $(\sigma(n),\ldots,\sigma(1))=\sigma\cdot(n,n-1,\ldots,1)$ is distributed according to $\pi_{n,1/q}$.
This provides a natural symmetry between Mallows permutations with parameters $q$ and $1/q$ and allows us to reduce the study of such permutations to $q\in[0,1]$.
Furthermore, the case $q=1$ corresponds to uniformly sampled permutations and thus tends to have very different properties from the case $q\neq 1$.
For these two reasons, we limit our study to the case $q\in[0,1)$ for the rest of the article (except in Section~\ref{sec:ledwood} where we briefly describe how to adapt our results to the case $q>1$).

The Mallows distribution happens to be the unique distribution on $\Sn$ which is $q$-invariant, that is such that
\begin{align*}
    \pi_{n,q}\big(\sigma\cdot(i~i+1)\big)=q^{\sign[\sigma(i+1)-\sigma(i)]}\cdot\pi_{n,q}(\sigma)\,,
\end{align*}
where $(i~i+1)$ is the transposition sending $i$ to $i+1$ and vice-versa, and $\sign[x]$ is the sign of $x$.
In fact, the $q$-invariant property happens to be a natural way to generalize Mallows permutations to one-sided infinite sets such as $\bN$~\cite{gnedin2010q} and two-sided infinite sets such as $\bZ$~\cite{gnedin2012two}.
We further describe the construction of one-sided and two-sided infinite Mallows permutations in Section~\ref{sec:mallows}.

\medskip

Let $\BST=\{\root\}\cup\bigcup_{k\geq1}\{\lv,\rv\}^k$ be the infinite binary tree seen as words on the alphabet $\{\lv,\rv\}$, where $\root$ corresponds equivalently to the root of the tree and the empty word.
We further identify binary trees with subsets of $\BST$ rooted at $\root$ and connected in the graph sense, that is subsets $T\subseteq\BST$ stable by considering any prefix of a word.

For any finite sequence of distinct values $x=(x_1,\ldots,x_n)$, define $\bst{x}$ the binary search tree of $x$ as follows.
If $n=0$ and $x$ is empty, then $\bst{x}=\emptyset$.
Otherwise, if $n\geq1$, write $x_-=(x_i:x_i<x_1)$ (respectively $x_+=(x_i:x_i>x_1)$) for the sequence of numbers smaller (respectively larger) than $x_1$ in the same order as in $x$ and let
\begin{align}\label{eq:bst}
    \bst{x}:=\{\root\}\cup\lv\bst{x_-}\cup\rv\bst{x_+}\,;
\end{align}
see Figure~\ref{fig:bst} for an example of the construction of $\bst{x}$.
This definition straightforwardly implies that $|\bst{x}|=n$.
It is also worth noting that, if $x=(x_1,\ldots,x_n)$, then for any sequence $\Bar{x}=(x_1,\ldots,x_n,x_{n+1},\ldots,x_{n+m})$ extending $x$, we have $\bst{x}\subseteq\bst{\Bar{x}}$.
Using this remark, we naturally extend the previous notations to infinite sequences $x=(x_1,x_2,\ldots)$ by letting
\begin{align*}
    \bst{x}:=\bigcup_{n=1}^\infty\bst{\big(x_1,\ldots,x_n\big)}\,.
\end{align*}
Since $\bst{(x_1,\ldots,x_n)}\subseteq\bst{(x_1,\ldots,x_n,x_{n+1})}$, the previous definition can actually be understood as $\bst{x}$ being the limit of $(\bst{(x_1,\ldots,x_n)})_{n\in\bN}$.

\def\permutation{1/4,2/2,3/5,4/1,5/6,6/3}
\def\sizeperm{6}
\def\edges{4/2,4/5,2/1,5/6,2/3}
\def\minusedges{2/1,2/3}
\def\plusedges{5/6}
\def\picturescale{0.5}
\def\gridlw{0.04cm}
\def\gridcolour{block}
\def\axislw{0.05cm}
\def\axiscolour{block}
\def\axisextra{0.5}
\def\axisinnersep{0.1cm}
\def\axisscale{1}
\def\nodescale{1.5}
\def\treecolour{theme}
\def\treelw{0.05cm}
\def\nodetextscale{1}
\def\highlightcolour{white!60!lighttheme}
\def\highlightlw{0.8cm}
\def\highlighttextcolour{lighttheme}

\begin{figure}[htb]
    \centering
    \begin{tikzpicture}[scale=\picturescale]
        \foreach\i\j in \permutation {
            \node(\j) at (\j,\sizeperm-\i+1){};
        }
        \foreach\i\j in \minusedges {
            \draw[\highlightcolour, line width=\highlightlw, line cap=round] (\i.center) -- (\j.center);
        }
        \foreach\i\j in \plusedges {
            \draw[\highlightcolour, line width=\highlightlw, line cap=round] (\i.center) -- (\j.center);
        }
        \draw[\gridcolour, line width=\gridlw, line cap=round, opacity=0.2] (0,0) grid (\fpeval{\sizeperm+1},\fpeval{\sizeperm+1});
        \draw[->, \axiscolour, line width=\axislw, line cap=round] (0,\fpeval{\sizeperm+1}) -- (\fpeval{\sizeperm+1+\axisextra},\fpeval{\sizeperm+1});
        \node[\axiscolour, scale=\axisscale, inner sep=\axisinnersep, anchor=west] at (\fpeval{\sizeperm+1+\axisextra},\fpeval{\sizeperm+1}){$\sigma(i)$};
        \draw[->, \axiscolour, line width=\axislw, line cap=round] (0,\fpeval{\sizeperm+1}) -- (0,-\axisextra);
        \node[\axiscolour, scale=\axisscale, inner sep=\axisinnersep, anchor=north] at (0,-\axisextra){$i$};
        \foreach\i\j in \edges {
            \draw[\treecolour, line width=\treelw] (\i.center) -- (\j.center);
        }
        \foreach\i\j in \permutation {
            \node[\treecolour, draw, circle, scale=\nodescale, line width=\treelw, fill=white] at (\j,\sizeperm-\i+1){};
            \node[\treecolour, scale=\nodetextscale] at (\j,\sizeperm-\i+1){$\mathbf{\j}$};
        }
        \node[\highlighttextcolour, scale=\nodetextscale] at (1,1.5){$\bst{x_-}$};
        \node[\highlighttextcolour, scale=\nodetextscale] at (7,4.5){$\bst{x_+}$};
        \node[\highlighttextcolour, scale=\nodetextscale] at (5,6.5){$x_1$};
    \end{tikzpicture}
    \caption{
        A representation of the binary search tree $\bst{x}$ for the sequence $x=(4,2,5,1,6,3)$.
        The tree is represented in dark green and the numbers on the node are the values of the corresponding $x_i$.
        In light green we show the construction of the binary search tree as defined in~\eqref{eq:bst}: $x_1$ (here $x_1=4$) is placed at the root of the tree and the sequence is split into $x_-$ composed of values lower than $x_1$ (here $x_-=(2,1,3)$) at the left and $x_+$ composed of values larger higher than $x_1$ (here $x_+=(5,6)$) at the right.
        The root and both left and right binary search trees are then combined as in~\eqref{eq:bst} to obtain the full binary search tree.
    }
    \label{fig:bst}
\end{figure}
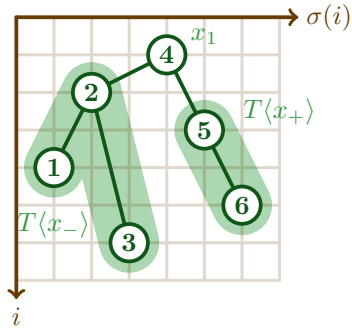

\medskip

Consider now trying to define binary search trees for a two-sided infinite sequence $x=(\ldots,x_{-1},x_0,x_1,\ldots)$.
A simple first case to consider would be $x$ such that $x_i=i$ for all $i\in\bZ$.
Since the binary search tree of the sequence $(1,2,\ldots)$ is the infinite right-downward path, we would naturally like the binary search tree of this $x$ to be the combination of the infinite right-downward and infinite left-upward paths.
We thus start by defining the notion of \textit{left-upward} in a binary tree.

Since right-downward paths are encoded by $\rv^k$ for $k\geq0$, we want to encode left-upward paths with $\rv^k$ for $k<0$.
Naturally, negative powers of words are not properly defined, but we can add a new symbol $\uv$ (\textit{up}silon) to our alphabet $\{\lv,\rv\}$ to describe $\rv^{-1}$.
From now on, we thus use the convention that
\begin{align}\label{eq:alphabet}
    \rv^k=\left\{\begin{array}{ll}
        \rv^k & \textrm{if $k>0$} \\
        \root & \textrm{if $k=0$} \\
        \uv^{|k|} & \textrm{if $k<0$}\,,
    \end{array}\right.
\end{align}
and we let the extension of the infinite binary tree be
\begin{align*}
    \RIBST:=\bigcup_{k\in\bZ}\{\rv^k\}\cup\rv^k\lv\BST=\BST\cup\left(\bigcup_{k\geq1}\{\uv^k\}\cup\uv^k\lv\BST\right)\,.
\end{align*}
We further let $\leftT:=\{\root\}\cup\lv\BST$ be the binary tree obtained by removing the first right edge of the complete binary tree;
this tree will play a key role in our study, in particular since its shifted sequence $(\rv^k\leftT)_{k\in\bZ}$ is a partition of $\RIBST$.

Let $x=(\ldots,x_{-1},x_0,x_1,\ldots)$ be a two-sided infinite sequence with distinct values.
Since the binary search tree of an infinite one-sided sequence is built as the limit of the tree obtained from the first $n$ indices, we want to define the binary search tree of $x$ as the limit of $\bst{x(n)}$ where $x(n)=\{x_{-n},x_{-n+1},\ldots,x_n\}$.
Unfortunately, this approach does not directly work, in particular since $x_{-n}$ plays too big a role in the structure of $\bst{x(n)}$.
However, by re-rooting $\bst{x(n)}$ according to some shift $s_n$ so that the tree is approximately as high as it is low, we obtain a sequence of trees $(\rv^{-s_n}\bst{x(n)})_{n\in\bN}$ which are not directly subtrees of each other, but have a growth-like behaviour.
Indeed, under mild assumptions on $x$, this sequence of trees eventually has a fixed structure on any ball of finite radius around the root $\root$.
This allows us to define $\ribst{x}\subseteq\RIBST$ as the limit of $(\rv^{-s_n}\bst{x(n)})_{n\in\bN}$ on any finite ball, corresponding to the generalization of binary search trees to two-sided infinite sequences.
We call $\ribst{x}$ the \textit{redwood tree} of $x$ and refer to Section~\ref{sec:redwood} for a more precise definition.

\medskip

Fix $q\in[0,1)$ and let $\imp$ be a two-sided infinite Mallows permutation on $\bZ$.
Write $Y_q=\ribst{\imp}$ for the redwood tree drawn from $\imp$ and let $\imt=(\imroot,Y_q)$ be the rooted tree where $\imroot$ is uniformly sampled from the set
\begin{align*}
    Y_q\cap\leftT=\{\root\}\cup\Big(Y_q\cap\lv\BST\Big)\,.
\end{align*}
In other words, $\imroot$ is a uniform element chosen from the set composed of the root of the tree and the left subtree of the root.
We now state our main theorem.

\begin{thm}[Local limit of Mallows trees]\label{thm:conv}
    Fix $q\in[0,1)$.
    For all $n\in\bN$, let $\mp{n}$ be distributed as $\pi_{n,q}$.
    Then we have
    \begin{align*}
        \bst{\mp{n}}\localconv\imt\,,
    \end{align*}
    where the convergence occurs according to the almost-sure local topology, defined in Section~\ref{sec:local}.
\end{thm}

The proof of Theorem~\ref{thm:conv} can be found in Section~\ref{sec:local limit}.
This result states an interesting limit swap between binary search tree structures and random permutation models;
indeed, it can be interpreted as ``the limit of the binary search tree of a finite Mallows permutation is the binary search tree of the infinite Mallows permutation''.
Another important fact related to this theorem is that, to the best of our knowledge, this is the first occurrence of a non-trivial almost-sure local convergence for naturally growing sequences of graphs.
This strong type of convergence arises from the the regenerative property of Mallows permutations and heavily relies on the strong law of large numbers.
To complement Theorem~\ref{thm:conv}, we now provide a clearer characterization of the distribution of $\imt$.
For the rest of this work, we write $\pgeom{p}$ for the geometric distribution with parameter $p$ and support $\bN=\{1,2,\ldots\}$, whereas $\geom{p}$ refers to the geometric distribution with $0$ in its support.
We further let $\sbgeom{p}$ be the sized biased version of $\geom{p}$ subtracted by $1$.
All these distributions are summarized by
\begin{align}\label{eq:all geoms}
    \prob(G=k)=\left\{\begin{array}{ll}
        (1-p)^kp & \textrm{for $k\geq0$, if $G\sim\geom{p}$} \\
        (1-p)^{k-1}p & \textrm{for $k\geq1$, if $G\sim\pgeom{p}$} \\
        (k+1)(1-p)^kp^2 & \textrm{for $k\geq0$, if $G\sim\sbgeom{p}$} \\
    \end{array}\right.
\end{align}

\begin{thm}[Distribution of the redwood Mallows tree]\label{thm:imt}
    Fix $q\in[0,1)$ and let $Y_q$ be the redwood tree drawn from a two-sided infinite Mallows permutation.
    For any $k\in\bZ$, let
    \begin{align*}
        Y_q(k)=Y_q\cap\rv^k\lv\BST=\big(Y_q\cap\rv^k\leftT\big)\setminus\{\rv^k\}
    \end{align*}
    be the left subtree of the node $\rv^k$.
    Then $(Y_q(k))_{k\in\bZ}$ is a two-sided sequence of independent Mallows trees with parameter $q$ and size $\geom{1-q}$, except for $Y_q(0)$ having size $\sbgeom{1-q}$.
    Moreover, the structure of the tree is independent given the size, meaning that for $n\geq0$ and $t$ of size $n$, we have
    \begin{align*}
        \prob\Big(Y_q(k)=\rv^k\lv t~\Big|~\big|Y_q(k)\big|=n\Big)=\prob\Big(\bst{\mp{n}}=t\Big)\,,
    \end{align*}
    where $\mp{n}$ is $\pi_{n,q}$-distributed.
    Conversely, these properties fully characterize the distribution of $Y_q$.
\end{thm}

The proof of Theorem~\ref{thm:imt} can be found in Section~\ref{sec:imt}.
The distribution of the left subtrees in $Y_q$ follows from the global structure of Mallows trees as studied in~\cite[Section~7]{evans2012trickle}, in particular the fact that the subtree of a Mallows tree is, conditionally given its size, a Mallows tree itself.
Similarly, the fact that their sizes are geometrically distributed follows from the distribution of the distance between two consecutive records in a Mallows permutations and was already observed in the case of the one-sided infinite Mallows distribution in~\cite[Lemma~1.10]{addario2021height}.
What might come as a surprise however, is the size-biased distribution of the left subtree of the root.
This can be understood by seeing $Y_q$ as the limiting object of $\bst{\mp{n}}$.
Indeed, since the sizes of the left subtrees in a finite Mallows tree converge to a sequence of independent geometric random variables, the size of the left subtree containing a uniformly chosen node would have the size-biased version of the geometric distribution.
The fact that $Y_q$, defined as the redwood tree drawn from the two-sided infinite Mallows distribution, also has this property arises from the fact that the size-biased geometric distribution subtracted by $1$ and the sum of two independent geometric random variables have the same distribution.
More details regarding this property and its relation to the proof of Theorem~\ref{thm:imt} can be found in Section~\ref{sec:imt} and is in particular shown in~\eqref{eq:geomsing}.

\medskip

We conclude this section with three other limits of Mallows trees, depending on the topology considered.
The first one is the \textit{local rooted convergence}, or simply \textit{rooted convergence}: if instead of considering the local behaviour of a uniformly sampled node in the tree, we fix one (the root), and consider the structure of the tree as seen from this node.
This yields the classical binary search tree drawn from a Mallows permutation on $\bN$, as stated below.

\begin{thm}[Rooted limit of Mallows trees]\label{thm:rooted}
    Fix $q\in[0,1)$.
    For any $n\in\bN$, let $\mp{n}$ be $\pi_{n,q}$-distributed.
    Then we have
    \begin{align*}
        \bst{\mp{n}}\rootconv\bst{\mp{\infty}}\,,
    \end{align*}
    where $\mp{\infty}$ is a Mallows permutation on $\bN$ with parameter $q$ and the convergence occurs according to the rooted topology, defined in Section~\ref{sec:other local}.
\end{thm}

The second one is the strong Gromov–Hausdorff–Prokhorov (GHP) topology and corresponds to the distributional limit of the tree seen as a metric space.
After proper scaling, arising from the height of the tree as shown in~\cite{addario2021height}, this yields a rather simple metric space with uniform measure on $[0,1]$.

\begin{thm}[Strong GHP limit of Mallows trees]\label{thm:ghp}
    Fix $q\in[0,1)$.
    For any $n\in\bN$, let $\mp{n}$ be $\pi_{n,q}$-distributed.
    Let $d_n$ be the graph distance in $\bst{\mp{n}}$ and $\mu_n$ be the uniform distribution on the nodes of $\bst{\mp{n}}$.
    Then we have
    \begin{align*}
        \Big(\bst{\mp{n}},\big((1-q)n\big)^{-1}\cdot d_n,\mu_n\Big)\ghpconv\ghplimit\,,
    \end{align*}
    where the convergence occurs according to the strong GHP topology, defined in Section~\ref{sec:ghp}, and
    \begin{align*}
        \ghplimit=\Big([0,1],d,\mu_{[0,1]}\Big)
    \end{align*}
    with $d(x,y)=|x-y|$ and $\mu_{[0,1]}$ being the uniform measure on $[0,1]$.
\end{thm}

The third and last one is according to the newly introduced \textit{subtree size topology}~\cite{grubel2023note} and corresponds to the convergence of the ratio of nodes in each subtree.
Since left hanging subtrees have geometric sizes in a Mallows tree, this yields the infinite rightward path.

\begin{thm}[Subtree size limit of Mallows trees]\label{thm:ssc}
    Fix $q\in[0,1)$.
    For any $n\in\bN$, let $\mp{n}$ be $\pi_{n,q}$-distributed.
    Then we have
    \begin{align*}
        \bst{\mp{n}}\subtreeconv\psi_R\,,
    \end{align*}
    where
    \begin{align}\label{eq:def psiR}
        \psi_R(v)=\left\{
        \begin{array}{ll}
            1 & \textrm{if $v\in\{\rv^k:k\geq0\}$}\\
            0 & \textrm{otherwise}
        \end{array}\right.
    \end{align}
    and the convergence occurs according to the subtree size topology, defined in Section~\ref{sec:ssc}.
\end{thm}

The proofs of Theorem~\ref{thm:rooted}, Theorem~\ref{thm:ghp}, and Theorem~\ref{thm:ssc} can respectively be found in Section~\ref{sec:rooted}, Section~\ref{sec:ghp}, and Section~\ref{sec:ssc}.

\section{Notations and background work}\label{sec:notations}

In this section, we precise the definitions of the objects introduced in Section~\ref{sec:main} and introduce further notations and results useful for the proofs of the different theorems.

\subsection{Mallows permutations}\label{sec:mallows}

In this section, we describe useful constructions of finite, one-sided, and two-sided Mallows permutations.
These constructions will prove to be useful when studying the properties of the related binary search trees.

\subsubsection{Finite and one-sided Mallows permutations}\label{sec:finite mallows}

There exists various methods to sample Mallows permutations and, since we are interested in the limit of objects related to Mallows permutations with increasing size $n$, it is worth considering an adapted coupling.
For this reason, we consider the \textit{infinite Bernoulli model} as developed in~\cite[Section~1.4]{addario2021height} which we explain below.

Consider $B=(B_{i,j})_{i,j\geq1}$ an infinite matrix with independent $\bern{1-q}$ entries.
Let $I_1,I_2,\ldots$ be the first non-zero entry of each row not already chosen,
that is
\begin{align*}
    I_i=\min\Big\{j\notin\{I_1,\ldots,I_{i-1}\}:B_{i,j}=1\Big\}\,,
\end{align*}
and let $\mp{n}$ be the permutation corresponding to the ordering of $(I_1,\ldots,I_n)$.
Then, thanks to~\cite[Proposition~1.8]{addario2021height}, $\mp{n}$ is $\pi_{n,q}$-distributed.
Moreover, the sequence $\mp{\infty}=(I_1,I_2,\ldots)$ actually corresponds to a permutation, since every column of $B$ has at least one non-zero entry, and is distributed as the Mallows distribution on $\bN$~\cite[Definition~4.4]{gnedin2010q}.
From now on, we assume that $\mp{n}$, for $n\in\bN\cup\{\infty\}$, refers to the random variable obtained using the coupling from the infinite Bernoulli model.

We now state some useful notations and results from~\cite{addario2021height}.
First, let
\begin{align}\label{eq:records}
    \reco{n}:=\Big|\Big\{j:\forall i<j,\mp{n}(j)>\mp{n}(i)\Big\}\Big|=\Big|\Big\{j:\forall i<j,I_j>I_i\Big\}\Big|
\end{align}
be the number of records of $\mp{n}$ and
\begin{align*}
    \free{n}:=\max\big\{I_1,\ldots,I_n\big\}-n
\end{align*}
be the number of \textit{free spaces} in the tree $\bst{\mp{n}}$.
We refer to $\free{n}$ as the number of free spaces since it corresponds to the number of elements added to the left of $\rv^{\reco{n}}$ between $\bst{\mp{n}}$ and $\bst{\mp{\infty}}$;
indeed, a better definition of $\free{n}$ is
\begin{align}\label{eq:spaces}
    \free{n}=\left|\Big(\bst{\mp{\infty}}\cap\bigcup_{0\leq k\leq\reco{n}}\rv^k\leftT\Big)\setminus\bst{\mp{n}}\right|\,.
\end{align}
We refer to Figure~\ref{fig:ibm} for a visual representation of this property.
Interestingly, the previous equation implies that the number of left hanging subtrees in $\bst{\mp{n}}$ that are not full yet can be bounded by the number of free spaces
\begin{align}\label{eq:bound left}
    \left|\Big\{0\leq k\leq\reco{n}:\bst{\mp{\infty}}\cap\rv^k\leftT\neq\bst{\mp{n}}\cap\rv^k\leftT\Big\}\right|\leq\free{n}\,,
\end{align}
bound which will be useful in proving Theorem~\ref{thm:conv}.

\def\drawingxshift{11cm}
\def\drawingyshift{13cm}
\def\infocolour{black}
\def\infoinnersep{0.2cm}
\def\treehiddencolour{\highlightcolour}
\newcommand{\drawpermutation}[5]{
    \begin{scope}[xshift=#4*\drawingxshift,yshift=#5*\drawingyshift]
        \foreach\i\j in \permutation {
            \node(\j) at (\j,\sizeperm-\i+1){};
        }
        \draw[\gridcolour, line width=\gridlw, line cap=round, opacity=0.2] (0,0) grid (\fpeval{\sizeperm+1},\fpeval{\sizeperm+1});
        \draw[->, \axiscolour, line width=\axislw, line cap=round] (0,\fpeval{\sizeperm+1}) -- (\fpeval{\sizeperm+1},\fpeval{\sizeperm+1});
        \draw[->, \axiscolour, line width=\axislw, line cap=round] (0,\fpeval{\sizeperm+1}) -- (0,0);
        \foreach\i\j in \edges {
            \draw[\treehiddencolour, line width=\treelw] (\i.center) -- (\j.center);
        }
        \foreach\i\j in \permutation {
            \node[\treehiddencolour, draw, circle, scale=\nodescale, line width=\treelw, fill=white] at (\j,\sizeperm-\i+1){};
            \node[\treehiddencolour, scale=\nodetextscale] at (\j,\sizeperm-\i+1){$\mathbf{\j}$};
        }
        \foreach\i\j in #2 {
            \draw[\treecolour, line width=\treelw] (\i.center) -- (\j.center);
        }
        \foreach\i\j in #1 {
            \node[\treecolour, draw, circle, scale=\nodescale, line width=\treelw, fill=white] at (\j,\sizeperm-\i+1){};
            \node[\treecolour, scale=\nodetextscale] at (\j,\sizeperm-\i+1){$\mathbf{\j}$};
        }
        \node[\infocolour, scale=\nodetextscale, inner sep=\infoinnersep, anchor=north] at (\fpeval{\sizeperm/2+1/2},0){#3};
    \end{scope}
}

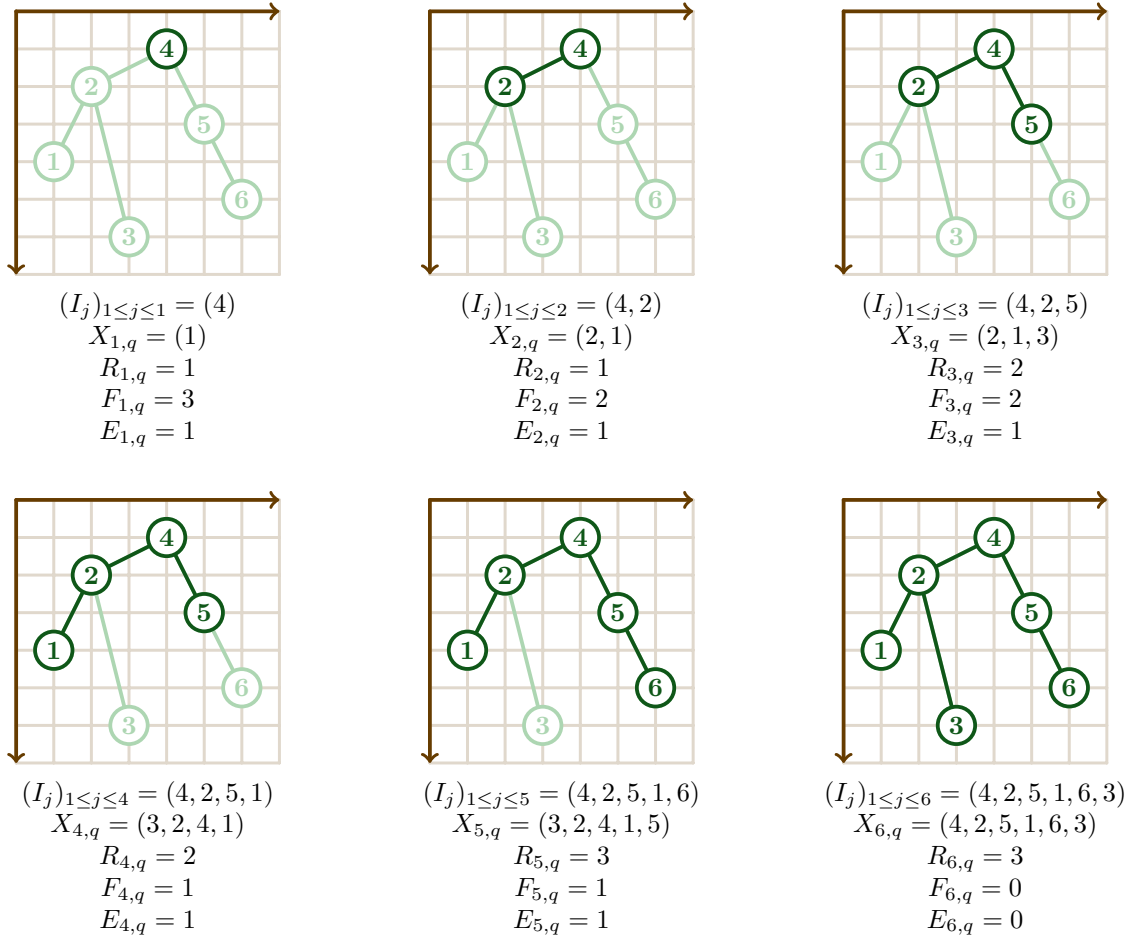
\begin{figure}[htb]
    \centering
    \begin{tikzpicture}[scale=\picturescale]
        \drawpermutation
        {{1/4}}
        {{}}
        {\begin{tabular}{c}
            $(I_j)_{1\leq j\leq1}=(4)$ \\
            $\mp{1}=(1)$ \\
            $\reco{1}=1$ \\
            $\free{1}=3$ \\
            $E_{1,q}=1$
        \end{tabular}}
        {0}{0}
        \drawpermutation
        {{1/4,2/2}}
        {{4/2}}
        {\begin{tabular}{c}
            $(I_j)_{1\leq j\leq2}=(4,2)$ \\
            $\mp{2}=(2,1)$ \\
            $\reco{2}=1$ \\
            $\free{2}=2$ \\
            $E_{2,q}=1$
        \end{tabular}}
        {1}{0}
        \drawpermutation
        {{1/4,2/2,3/5}}
        {{4/2,4/5}}
        {\begin{tabular}{c}
            $(I_j)_{1\leq j\leq3}=(4,2,5)$ \\
            $\mp{3}=(2,1,3)$ \\
            $\reco{3}=2$ \\
            $\free{3}=2$ \\
            $E_{3,q}=1$
        \end{tabular}}
        {2}{0}
        \drawpermutation
        {{1/4,2/2,3/5,4/1}}
        {{4/2,4/5,2/1}}
        {\begin{tabular}{c}
            $(I_j)_{1\leq j\leq4}=(4,2,5,1)$ \\
            $\mp{4}=(3,2,4,1)$ \\
            $\reco{4}=2$ \\
            $\free{4}=1$ \\
            $E_{4,q}=1$
        \end{tabular}}
        {0}{-1}
        \drawpermutation
        {{1/4,2/2,3/5,4/1,5/6}}
        {{4/2,4/5,2/1,5/6}}
        {\begin{tabular}{c}
            $(I_j)_{1\leq j\leq5}=(4,2,5,1,6)$ \\
            $\mp{5}=(3,2,4,1,5)$ \\
            $\reco{5}=3$ \\
            $\free{5}=1$ \\
            $E_{5,q}=1$
        \end{tabular}}
        {1}{-1}
        \drawpermutation
        {{1/4,2/2,3/5,4/1,5/6,6/3}}
        {{4/2,4/5,2/1,5/6,2/3}}
        {\begin{tabular}{c}
            $(I_j)_{1\leq j\leq6}=(4,2,5,1,6,3)$ \\
            $\mp{6}=(4,2,5,1,6,3)$ \\
            $\reco{6}=3$ \\
            $\free{6}=0$ \\
            $E_{6,q}=0$
        \end{tabular}}
        {2}{-1}
    \end{tikzpicture}
    \caption{
        The representation of the first few number of records and free spaces for the sequence $(I_j)_{j\geq1}=(4,2,5,1,6,3,\ldots)$ along with the corresponding permutations and binary search trees.
        We further provide the value of $E_{n,q}$, the number of left subtrees to be filled, as defined on the left hand side of~\eqref{eq:bound left}.
        From this representation, we verify the following properties.
        \textbf{(1)}~$\mp{n}$ is the ordering of the sequence $(I_1,\ldots,I_n)$.
        \textbf{(2)}~$\reco{n}$ is both the number of records of $\mp{n}$ and the number of nodes on the rightmost path of the tree.
        \textbf{(3)}~$\free{n}$ is both the difference between the maximal value of $(I_1,\ldots,I_n)$ and $n$ and the number of nodes to be filled on the left part of the tree (i.e. the number of shaded green nodes to the left of the rightmost dark green node).
        \textbf{(4)}~$E_{n,q}\leq\free{n}$ as claimed in~\eqref{eq:bound left} since the number of incomplete left subtrees is less than the number of left nodes to be filled.
    }
    \label{fig:ibm}
\end{figure}

We now provide two important results on $\reco{n}$ and $\free{n}$.
The first one states the asymptotic behaviour of $\reco{n}$ and has strong enough tail bounds to provide an almost-sure convergence after re-scaling.

\begin{lemma}\label{lem:records}
    Fix $q\in[0,1)$.
    Let $\reco{n}$ be the number of records of $\mp{n}$ coupled according to the infinite Bernoulli model.
    Then we have
    \begin{align*}
        \frac{\reco{n}-n(1-q)}{\sqrt{n}\log n}\asconv0\,.
    \end{align*}
    In particular, this implies that $\reco{n}/n\asconv(1-q)$
\end{lemma}

\begin{proof}
    Using~\cite[Lemma~2.4]{addario2021height}, we know that, for any $\epsilon>0$ and $n\in\bN$, we have
    \begin{align*}
        \prob\left(\left|\frac{\reco{n}}{\mu(n,q)}-1\right|>\epsilon\right)&\leq\exp\left(-\left[1-(1+\epsilon)\log\left(\frac{e}{1+\epsilon}\right)\right]\mu(n,q)\right)\\
        &\pskip+\exp\left(-\left[1-\frac{1}{1+\epsilon}\log\big((1+\epsilon)e\big)\right]\mu(n,q)\right)\,,
    \end{align*}
    where
    \begin{align*}
        \mu(n,q)=\sum_{k=2}^n\frac{1-q}{1-q^k}\,.
    \end{align*}
    Now, let $\lambda(\epsilon)$ be the exponential power in the previous upper bound, that is
    \begin{align*}
        \lambda(\epsilon)=\min\left\{1-(1+\epsilon)\log\left(\frac{e}{1+\epsilon}\right),1-\frac{1}{1+\epsilon}\log\big((1+\epsilon)e\big)\right\}\,,
    \end{align*}
    and observe that $\lambda(\epsilon)\sim\epsilon^2/2$ as $\epsilon\rightarrow0$.
    Combining the previous results, we have that
    \begin{align*}
        \prob\left(\left|\frac{\reco{n}-n(1-q)}{\sqrt{n}\log n}\right|>\delta\right)&=\prob\left(\left|\frac{\reco{n}}{\mu(n,q)}-\frac{n(1-q)}{\mu(n,q)}\right|>\delta\cdot\frac{\sqrt{n}\log n}{\mu(n,q)}\right)\\
        &\leq\prob\left(\left|\frac{\reco{n}}{\mu(n,q)}-1\right|>\delta\cdot\frac{\sqrt{n}\log n}{\mu(n,q)}-\left|1-\frac{n(1-q)}{\mu(n,q)}\right|\right)\\
        &\leq2\exp\left(-\lambda\left(\delta\cdot\frac{\sqrt{n}\log n}{\mu(n,q)}-\left|1-\frac{n(1-q)}{\mu(n,q)}\right|\right)\cdot\mu(n,q)\right)\,.
    \end{align*}
    But now, thanks to~\cite[Proposition~2.5]{addario2021height}, we know that $\mu(n,q)=n(1-q)+O(\log n)$, which implies that
    \begin{align*}
        \delta\cdot\frac{\sqrt{n}\log n}{\mu(n,q)}-\left|1-\frac{n(1-q)}{\mu(n,q)}\right|&=\frac{\delta}{\sqrt{1-q}}\cdot\frac{\log n}{\sqrt{\mu(n,q)}}+o\left(\frac{\log n}{\sqrt{\mu(n,q)}}\right)\,.
    \end{align*}
    To conclude the proof of the lemma, recall that $\lambda(\epsilon)\sim\epsilon^2/2$ when $\epsilon\rightarrow0$ to obtain
    \begin{align*}
        \prob\left(\left|\frac{\reco{n}-n(1-q)}{\sqrt{n}\log n}\right|>\delta\right)&\leq\exp\left(-\frac{\delta^2\log^2n}{2(1-q)}+o(\log^2n)\right)
    \end{align*}
    and the almost-sure convergences follows from applying the Borel-Cantelli lemma.
\end{proof}

The second important result is related to $\free{n}$ and proves an other almost-sure convergence result.

\begin{lemma}\label{lem:spaces}
    Fix $q\in[0,1)$.
    Let $\free{n}$ be the number of free spaces as defined in~\eqref{eq:spaces}.
    Then, for any $(a_n)_{n\geq1}$ such that $a_n/\log n\rightarrow\infty$, we have
    \begin{align*}
        \frac{\free{n}}{a_n}\asconv0\,.
    \end{align*}
    In particular, this implies that
    \begin{align*}
        \frac{\free{n}}{n}\asconv0\,.
    \end{align*}
\end{lemma}

\begin{proof}
    Thanks to~\cite[Proposition~1.11]{addario2021height}, we know that the moment generating function of $\free{n}$ is
    \begin{align*}
        \expec\left[x^{\free{n}}\right]=\prod_{k=1}^n\frac{1-q^k}{1-q^kx}\,.
    \end{align*}
    Moreover, for any $k\geq1$ and $1<x<1/q$, we have
    \begin{align*}
        \frac{1-q^k}{1-q^kx}=1+(x-1)\frac{q^k}{1-q^kx}\leq1+(x-1)\frac{q^k}{1-qx}\,.
    \end{align*}
    Fix now $\epsilon>0$ and let $x=\sqrt{1/q}<1/q$.
    Then, using Chernoff's bounds along with the previous inequality, we have
    \begin{align*}
        \prob\left(\frac{\free{n}}{a_n}>\epsilon\right)&\leq x^{-a_n\epsilon}\expec\left[x^{\free{n}}\right]\leq q^\frac{a_n\epsilon}{2}\prod_{k=1}^n\left[1+q^{k-1/2}\right]\,,
    \end{align*}
    where we used that $(x-1)/(1-qx)=(1/\sqrt{q}-1)/(1-\sqrt{q})=1/\sqrt{q}$.
    Finally, using the convexity of exponential, we obtain
    \begin{align*}
        \prob\left(\frac{\free{n}}{a_n}>\epsilon\right)&\leq\exp\left(-a_n\cdot\frac{\epsilon|\log q|}{2}+\sum_{k=1}^nq^{k-1/2}\right)\leq\exp\left(-a_n\cdot\frac{\epsilon|\log q|}{2}+\frac{\sqrt{q}}{1-q}\right)\,,
    \end{align*}
    which, combined with the Borel-Cantelli lemma, proves the desired almost-sure convergence.
\end{proof}

\subsubsection{Two-sided infinite Mallows permutations}\label{sec:infinite mallows}

Using the results from~\cite{gnedin2012two}, we provide a construction of Mallows permutations on $\bZ$.
A permutation on $\bZ$ is said to be \textit{admissible} if
\begin{align}\label{eq:admissible}
    \big|\big\{i\in\bN:\sigma(i)\notin\bN\big\}\big|+\big|\big\{i\in\bZ\setminus\bN:\sigma(i)\in\bN\big\}\big|<\infty\,.
\end{align}
This property plays a key role here, since Mallows permutations on $\bZ$ are almost-surely admissible, and since admissibility will be used as a condition for the construction of redwood trees in Section~\ref{sec:redwood}.

Let $\mp{+},\mp{-}$ be two independent Mallows permutations on $\bN$.
Let $M=(M_1,M_2,\ldots)$ be a random sequence of independent non-negative integers such that, for all $i\geq1$, $M_i$ is distributed as a $\geom{1-q^i}$, that is, for all $m\in\bN$
\begin{align*}
    \prob(M_i=m)=q^{im}(1-q^i)\,.
\end{align*}
Note that this definition implies that $\sum_iM_i$ is almost-surely finite.
Let now $\Lambda=\Lambda(M)=(\Lambda_1\geq\Lambda_2\geq\ldots)$ be an infinite decreasing sequence such that, for all $i\geq1$, we have
\begin{align*}
    M_i=\big|\big\{j\geq1:\Lambda_j=i\big\}\big|\,.
\end{align*}
Using the notations of~\cite[Chapter I]{macdonald1998symmetric}, $\Lambda$ is a \textit{partition} and $M$ is the corresponding sequence of \textit{multiplicities};
in particular, $\Lambda$ is a random partition such that, for any $\lambda=(\lambda_1,\lambda_2,\ldots)$, we have
\begin{align*}
    \prob(\Lambda=\lambda)&=q^{|\lambda|}\cdot\prod_{k=1}^\infty(1-q^k)\,,
\end{align*}
where $|\lambda|=\sum_{k=1}^\infty\lambda_k$, as provided in~\cite[(8)]{gnedin2012two}.
We now combine the triplet $(\mp{+},\mp{-},\Lambda)$ to create a permutation as follows.

Given $\Lambda$, let $\ell_i=i-\Lambda_{i}$ and note that $\ell_1<\ell_2<\ldots$.
Further define $k_0>k_{-1}>k_{-2}>\ldots$ such that $\{\ell_i:i\in\bN\}$ and $\{k_i:i\in\bZ\setminus\bN\}$ form a partition of $\bZ$.
Now define $\imp$ as a function of $(\mp{+},\mp{-},\Lambda)$ by
\begin{align}\label{eq:inf perm}
    \imp(x)=\left\{\begin{array}{ll}
        \mp{+}(i) & \textrm{if $x=\ell_i$} \\
        1-\mp{-}(1-i) & \textrm{if $x=k_i$}
    \end{array}\right.
\end{align}
and note that $\imp$ is indeed a permutation of $\bZ$.
Thanks to~\cite[Section~3]{gnedin2012two}, $\imp$ is distributed as a Mallows permutation on $\bZ$, in particular by uniqueness of the $q$-invariant distribution.

We note here that the \textit{balanced} property (also introduced in~\cite{gnedin2012two}) does not play a role when considering the corresponding binary search trees.
So, although Mallows permutations on $\bZ$ are also almost-surely balanced, this property will not be mentioned in the remainder of this work;
the focus here is placed on admissible permutations, for which we can define the corresponding redwood tree.

\subsection{Redwood trees}\label{sec:redwood}

In this section, we properly define redwood trees to complete their introduction from Section~\ref{sec:main}.
We recall the convention on the alphabet $\{\lv,\rv,\uv\}$ as introduced in~\eqref{eq:alphabet} and that
\begin{align*}
    \RIBST=\bigcup_{k\in\bZ}\rv^k\leftT\,,
\end{align*}
where $\leftT=\{\root\}\cup\lv\BST$.
We also recall that $\sigma:\bZ\rightarrow\bZ$ is said to be admissible if it satisfies~\eqref{eq:admissible}.
We start by providing an important property of admissible functions.
A record index of $\sigma$ is a an index $r$ such that, for all $k<r$, we have $\sigma(r)>\sigma(k)$;
given a record index $r$, $\sigma(r)$ is the corresponding record.

\begin{lemma}\label{lem:admissible}
    Let $\sigma:\bZ\rightarrow\bZ$ be an admissible injective function.
    Then, for any $k\in\bZ$, there exists two record indices $r_-$ and $r_+$ such that $r_-<k<r_+$.
\end{lemma}

\begin{proof}
    Let $\sigma:\bZ\rightarrow\bZ$ be injective and satisfying~\eqref{eq:admissible}.
    Fix $k\in\bZ$.
    Using that $\sigma$ is admissible, we first have that
    \begin{align*}
        M_k=\max\big\{\sigma(\ell):\ell\leq k\big\}<\infty\,.
    \end{align*}
    But then, again thanks to $\sigma$ being admissible, we know that
    \begin{align*}
        r_+=\min\big\{\ell>k:\sigma(\ell)>M_k\big\}
    \end{align*}
    is well-defined and is a record index on the right of $k$.
    Similarly, the admissibility property of $\sigma$ gives us that
    \begin{align*}
        m_k=\max\big\{\sigma(\ell):\ell<k\big\}<\infty\,,
    \end{align*}
    which then, combined with the fact that $\sigma$ is injective and thus admits an inverse $\sigma^{-1}:\sigma(\bZ)\rightarrow\bZ$, gives us that
    \begin{align*}
        r_-=\sigma^{-1}(m_k)
    \end{align*}
    is well-defined and is a record index on the left of $k$.
\end{proof}

Let $\sigma$ be an admissible permutation of $\bZ$.
Using Lemma~\ref{lem:admissible}, the set of record indices of $\sigma$ is a two-sided infinite sequence.
Thus, we say that the sequence $r=(\ldots,r_{-1},r_0,r_1,\ldots)$ is the \textit{standard record representation} of $\sigma$ if and only if
\begin{itemize}
    \item $\{r_k:k\in\bZ\}$ is the complete set of record indices of $\sigma$;
    \item the sequence is in strictly increasing order; and
    \item it is (uniquely) indexed so that $\sigma(r_0)\leq 0<\sigma(r_1)$.
\end{itemize}
When the index $0$ of the sequence is not completely clear, we will underline it: $(\ldots,r_{-1},\underline{r_0},r_1,\ldots)$.

Given $r=(\ldots,r_{-1},r_0,r_1,\ldots)$ the standard record representation of $\sigma$, we let $\sigma[k]=(\sigma(i):\sigma(r_k)<\sigma(i)<\sigma(r_{k+1}))$ be the sequence of images of the permutation between two consecutive records, in the same order as in $\sigma$;
it is worth noting that this is well-defined since, for two record indices $r<r'$, we also have $\sigma(r)<\sigma(r')$.
Finally, we define the redwood tree of $\sigma$ to be
\begin{align}\label{eq:redwood}
    \ribst{\sigma}:=\bigcup_{k\in\bZ}\{\rv^k\}\cup\rv^k\lv\bst{\sigma[k]}\,.
\end{align}
From this definition, it is clear that $\ribst{\sigma}\subseteq\RIBST$ and we now explore further properties of this object.

\begin{lemma}\label{lem:ribst and bst}
    Let $\sigma$ be a permutation on $\bN$.
    Let $\Tilde{\sigma}$ be the extension of $\sigma$ to $\bZ$ by the identity, that is
    \begin{align*}
        \Tilde{\sigma}(k)=\left\{\begin{array}{ll}
            \sigma(k) & \textrm{if $k\in\bN$}\\
            k & \textrm{otherwise}\,.
        \end{array}\right.
    \end{align*}
    Then, we have
    \begin{align*}
        \ribst{\Tilde{\sigma}}=\bst{\sigma}\cup\bigcup_{k<0}\{\rv^k\}=\bst{\sigma}\cup\bigcup_{k\in\bN}\{\uv^k\}\,,
    \end{align*}
    which implies that $\ribst{\Tilde{\sigma}}\cap\BST=\bst{\sigma}$.
\end{lemma}

\begin{proof}
    First, note that, if $r_1<r_2<\ldots$ is the ordered sequence of record indices of $\sigma$, then it can be used to construct the binary search tree of $\sigma$ as follows
    \begin{align*}
        \bst{\sigma}=\bigcup_{k\geq0}\{\rv^k\}\cup\rv^k\lv\bst{\sigma[k]}\,,
    \end{align*}
    where $\sigma[k]$ is defined as the sequence of images between the two consecutive records $\sigma(r_k)$ and $\sigma(r_{k+1})$ (using the convention that $r_0=0$ and $\sigma(r_0)=0$).
    Now, it suffices to see that the standard record representation of $\Tilde{\sigma}$ is $(\ldots,-2,-1,\underline{0},r_1,r_2,\ldots)$ to obtain that $\Tilde{\sigma}[k]$ is equal to $\sigma[k]$ when $k\geq0$ and is the empty permutation otherwise.
    The result then follows from the definition of the redwood tree and the previous identity on the binary search tree of $\sigma$.
\end{proof}

The previous lemma provides a first reason as to why the redwood tree can be considered as the extension of the binary search tree to two-sided admissible permutations.
However, another characteristic relation between redwood trees and binary search trees is given in the following proposition.
For $a,b\in\bZ$, we let $\llbracket a,b\rrbracket=\{a,a+1,\ldots,b-1,b\}$ be the set of integers between $a$ and $b$.

\begin{prop}\label{prop:local redwood}
    Let $\sigma$ be an admissible permutation on $\bZ$ and $r=(\ldots,r_{-1},r_0,r_1,\ldots)$ be its standard record representation.
    For any $n\in\bN$, let $\sigma_n=(\sigma(-n),\ldots,\sigma(n-1),\sigma(n))$ be the restriction of $\sigma$ to $\llbracket -n,n\rrbracket$ and let $s_n$ be the number of records of $\sigma_n$ smaller or equal than $0$.
    Then, for any finite $I\subseteq\bZ$, there exists $n_0$ such that, for all $n\geq n_0$ we have
    \begin{align*}
        \ribst{\sigma}\cap\left(\bigcup_{k\in I}\rv^k\leftT\right)=\rv^{-s_n}\bst{\sigma_n}\cap\left(\bigcup_{k\in I}\rv^k\leftT\right)\,.
    \end{align*}
    In words, the redwood tree $\ribst{\sigma}$ and the shifted binary search tree $\rv^{-s_n}\bst{\sigma_n}$ are locally equal for $n$ large enough.
\end{prop}

\begin{proof}
    Without loss of generality, we assume that $I=\llbracket -K,K\rrbracket$ for some $K\geq1$.
    For any $n\in\bN$, let $r(n)=(r_{1-s_n}(n),\ldots,r_{-1}(n),r_0(n),\ldots,r_{m_n}(n))$ be the ordered set of record indices of $\sigma_n$ with the convention that $\sigma_n(r_0(n))\leq 0<\sigma_n(r_1(n))$ (so that $s_n$ is indeed the number of records of $\sigma_n$ smaller or equal than $0$).
    Finally, since $\sigma$ is a bijection of $\bZ$, define
    \begin{align*}
        n_0=\max\bigg\{|k|:k\in\sigma^{-1}\Big(\llbracket\sigma(r_{-K}),\sigma(r_{K+1})\rrbracket\Big)\bigg\}
    \end{align*}
    so that
    \begin{align*}
        \llbracket\sigma(r_{-K}),\sigma(r_{K+1})\rrbracket\subseteq\sigma\big(\llbracket-n_0,n_0\rrbracket\big)\,.
    \end{align*}
    In particular, this implies that $n_0\geq\max\{r_{K+1},-r_{-K}\}$.
    Note that $\sigma(\llbracket-n_0,n_0\rrbracket)$, or in general $\sigma(\llbracket-n,n\rrbracket)$, is not necessarily a segment of $\bZ$ itself;
    thus the previous inclusion simply means that it contains the segment defined by $\sigma(r_{-K})$ and $\sigma(r_{K+1})$ but may very well contain other isolated segments or singletons.
    We now prove that $n_0$ satisfies the desired condition.

    \medskip

    Let $n\geq n_0$.
    First, for any $-K\leq k\leq K+1$, thanks to the fact that $r_k$ is a record index of $\sigma$, it is also a record index of $\sigma_n$.
    Then, by the definition of both $r$ and $r(n)$, this implies that $r_k=r_k(n)$.
    Now, for any $-K\leq k\leq K$, using that $\sigma(r_k)<\sigma(r_{k+1})$, we have that
    \begin{align*}
        \llbracket\sigma(r_k),\sigma(r_{k+1})\rrbracket\subseteq\llbracket\sigma(r_{-K}),\sigma(r_{K+1})\rrbracket\subseteq\sigma\big(\llbracket-n_0,n_0\rrbracket\big)\subseteq\sigma\big(\llbracket-n,n\rrbracket\big)\,.
    \end{align*}
    Since $\sigma_n$ is the restriction of $\sigma$ to $\llbracket-n,n\rrbracket$, this also implies that
    \begin{align*}
        \llbracket\sigma_n(r_k(n)),\sigma_n(r_{k+1}(n))\rrbracket=\llbracket\sigma_n(r_k),\sigma_n(r_{k+1})\rrbracket\subseteq\sigma_n\big(\llbracket-n,n\rrbracket\big)\,.
    \end{align*}
    Finally, recall the definition of $\sigma[k]$ as the sequence of images of $\sigma$ between $\sigma(r_k)$ and $\sigma(r_{k+1})$ and similarly define $\sigma_n[k]$ to obtain that $\sigma[k]=\sigma_n[k]$ for any $-K\leq k\leq K$.

    Decompose $\bst{\sigma_n}$ using its records in the same way as in the proof of Lemma~\ref{lem:ribst and bst} to obtain
    \begin{align*}
        \bst{\sigma_n}&=\bigcup_{-s_n\leq k\leq m_n+1}\{\rv^{k+s_n}\}\cup\rv^{k+s_n}\lv\bst{\sigma_n[k]}\\
        &=\rv^{s_n}\bigcup_{-s_n\leq k\leq m_n+1}\{\rv^k\}\cup\rv^k\lv\bst{\sigma_n[k]}\,.
    \end{align*}
    Using the previously shown result stating that $\sigma_n[k]=\sigma[k]$ for $-K\leq k\leq K$, we have that
    \begin{align*}
        \rv^{-s_n}\bst{\sigma_n}\cap\left(\bigcup_{k\in I}\rv^k\leftT\right)&=\bigcup_{|k|\leq K}\{\rv^k\}\cup\rv^k\lv\bst{\sigma[k]}\\
        &=\ribst{\sigma}\cap\left(\bigcup_{k\in I}\rv^k\leftT\right)\,,
    \end{align*}
    where the last inequality directly follows from the definition of $\ribst{\sigma}$.
    This concludes the proof of the proposition by recalling that $n\geq n_0$ was arbitrarily chosen.
\end{proof}


\subsubsection{Disclaimer on the name and notations for redwood trees}\label{sec:disclaimer}

One might wonder why choosing to name the extension of binary search trees as \textit{redwood trees}.
A first version of the article actually referred to them as \textit{right infinite binary search trees}, since they are two-sided infinite extensions of binary search tree growing in a right-downward direction\footnote{This is naturally equivalent to growing in a left-upward direction in the case of a two-sided infinite tree, but since the idea was to extend right-downward paths in classical binary search trees, we used the downward direction as the standard one.}.
This name also implied the existence of \textit{left} infinite binary search tree, which are briefly discussed in Section~\ref{sec:ledwood}, and relate here to the limit of finite Mallows trees when $q>1$.

The name right infinite binary search tree, although easily abbreviated to RIBST, was a bit too long (the name \textit{binary search tree} is already rather long on its own) and felt like a precise description of an object rather than its actual name.
For this reason, and since these trees are both infinitely high and infinitely deep, we decided to name them based on the highest species of trees in the world: the \textit{sequoia sempervirens}, also called \textit{California redwoods}\footnote{https://en.wikipedia.org/wiki/Sequoia\_sempervirens}.

Following this logic, redwood trees could easily refer to both right and left infinite binary search trees.
However, the fact that both \textit{right infinite binary search trees} and \textit{redwood trees} start with an `r' convinced us of this choice;
thus, we invite anyone who would rather use left infinite binary search trees than right ones to refer to them as \textit{ledwood trees}.

Finally, and as far as it made sense in the context of this article, we always tried to use the letter $T$ when considering one-sided trees, that is trees with a natural node at the top, and the letter $Y$ for two-sided trees.
The reason for choosing $Y$ is that it can be interpreted as the letter $T$ where we extended its two top branches upwards;
with that logic, $Y$ can be used for both redwood and ledwood trees.

\subsection{Local convergence of graphs}\label{sec:local}

In this section, we provide some background on the topology of the local convergence of graphs.
Most of what is presented here can be found in more details in~\cite[Section~2]{van2023random}.

A graph $G=(V,E)$ is defined as a pair of vertices $V$ and edges $E$.
We write $|G|:=|V|$ for the size of $G$ and use the notation $v\in G$ to say that $v$ is a vertex from $V$.
Given a graph $G$ and one of its vertex $v\in G$, write $B_r(v,G)$ for the rooted ball of radius $r$ around $v$ in $G$;
for example $B_0(v,G)=(v,(\{v\},\emptyset))$.
Given two rooted graphs $(o,G)$ and $(o',G')$, where $G=(V,E)$ and $G'=(V',E')$, say that $(o,G)$ and $(o',G')$ are \textit{isomorphic} and write it $(o,G)\equiv(o',G')$ if there exists a bijection $\varphi:V\rightarrow V'$ such that $\varphi(o)=o'$ and
\begin{align*}
    (u,v)\in E~\Longleftrightarrow~(\varphi(u),\varphi(v))\in E'\,.
\end{align*}
We now use these notations to define almost-sure local convergence.

\medskip

Let $(G_n)_{n\geq1}$ be a sequence of finite random graphs and let $(o,G)$ be a (usually infinite) random rooted graph.
We say that $(G_n)_{n\geq1}$ converges locally almost-surely to $(o,G)$ and we write it $G_n\localconv(o,G)$ if, for any $r\geq1$ and any rooted graph $H$, we have
\begin{align}\label{eq:las}
    \frac{1}{|G_n|}\sum_{v\in G_n}\mathbbm{1}_{\{B_r(v,G_n)\equiv H\}}\overset{\mathrm{a.s.}}{\longrightarrow}\prob\big(B_r(o,G)\equiv H\big)\,,
\end{align}
where the convergence occurs according to the standard almost-sure convergence of random variables.
We note that the general definition of the local almost-sure convergence allows for the limit to be a random variable (ie a random measure on the space of rooted graphs) and we decided not to explicit this dependency for clarity and since all almost-sure local limits will be deterministic here;
we refer to~\cite[Section~2]{van2023random} for a more detailed discussion on the properties of local limits of random graphs.

\subsubsection{Other types of local convergence}\label{sec:other local}

For the sake of simplicity and since it suits our purpose, we directly stated the right type of convergence for the main result without explaining how this definition came to existence.
However, it is worth mentioning that there exists weaker and more common types of local convergence for graphs.
For example, one could replace the convergence in~\eqref{eq:las} by any other type of random variable convergence, for example in probability.
In particular, the most common type of local convergence is the \textit{local weak convergence}, introduced in~\cite{aldous2004objective,benjamini2001recurrence}, obtained by simply taking the expectation on both sides of~\eqref{eq:las}:
\begin{align*}
    \frac{1}{|G_n|}\sum_{v\in G_n}\prob\big(B_r(v,G_n)\equiv H\big)\longrightarrow\prob\big(B_r(o,G)\equiv H\big)\,;
\end{align*}
note that the convergence here is simply that of real numbers.

A natural way to further generalize the notion of local convergence is to restrict the distribution of $v$.
Indeed, considering the local weak convergence, the previous condition can be re-written as
\begin{align*}
    \prob\big(B_r(v,G_n)\equiv H\big)\longrightarrow\prob\big(B_r(o,G)\equiv H\big)\,,
\end{align*}
where $v$ is a uniformly sampled vertex of $G_n$.
One could however allow any type of distribution for $v$.
In particular, in the case of binary search tree, there exists a node with a special role: the root $\root$.
Therefore, we define the \textit{rooted local convergence} for binary search trees as follows.

Let $(T_n)_{n\geq1}$ be a sequence of finite binary search trees and let $T$ be a binary search tree.
We say that $T$ is the rooted (local) limit of $(T_n)_{n\geq1}$ and we write it $T_n\rootconv T$ if, for any $r\geq1$ and any rooted $t$, we have
\begin{align*}
    \prob\big(B_r(\root,T_n)\equiv t\big)\longrightarrow\prob\big(B_r(\root,T)\equiv t\big)\,.
\end{align*}
It is worth mentioning that this could also be called \textit{rooted local weak convergence} as it based on the convergence condition of the local weak convergence.
However, this would imply that both \textit{rooted local convergence in probability} or \textit{rooted local almost-sure convergence} would also make sense;
and while these two concepts could be defined using the same logic as before, these topologies actually simply state that $B_r(\root,T_n)=B_r(\root,T)$ for $n$ large enough, which slightly defeats the purpose of local limits (all other types of local convergence do not require any specific coupling of the random variables involved in the definitions).

\section{The almost-sure local convergence}\label{sec:as conv}

Since it is easier to prove the convergence result from Theorem~\ref{thm:conv} on the tree characterized in Theorem~\ref{thm:imt} than directly on the redwood tree drawn from the two-sided Mallows permutation, we start by proving the result of Theorem~\ref{thm:imt}.

\subsection{Records of the two-sided infinite Mallows tree}\label{sec:rec imt}

This section strongly relies on the notations introduced in Section~\ref{sec:infinite mallows} and, in particular, uses the construction of the two-sided infinite Mallows permutation $\imp$ from the triplet $(\mp{+},\mp{-},\Lambda)$.
We recall that $\{\ell_1<\ell_2<\ldots\}$ and $\{\ldots<k_{-2}<k_{-1}<k_0\}$ form a partition of $\bZ$, are defined by $\ell_i=1-\Lambda_i$, and are at the essence of the definition of $\imp$ as given in~\eqref{eq:inf perm}.

We start by providing an exact characterization of the standard record representation of $\imp$ given $\mp{+}$, $\mp{-}$, and $\Lambda$.
Write $(r_1,r_2,\ldots)$ for the the ordered sequence of record indices of $\mp{+}$ and define $(\rho_1,\rho_2,\ldots)$ inductively as follows.
First, $\rho_1=\mp{-}^{-1}(1)$ and then, given $\rho_i$, let
\begin{align*}
    \rho_{i+1}=\mp{-}^{-1}\Big(\min\big\{\mp{-}(k):k>\rho_i\big\}\Big)\,.
\end{align*}

\begin{prop}\label{prop:records imp}
    Using the previous notations, let $m\in\bN$ be minimal so that $k_{1-\rho_m}<\ell_1$.
    Then, the standard record representation of $\imp$ is given by
    \begin{align*}
        (\ldots,k_{1-\rho_{m+2}},k_{1-\rho_{m+1}},\underline{k_{1-\rho_m}},\ell_{r_1},\ell_{r_2},\ldots)\,.
    \end{align*}
\end{prop}

\begin{proof}
    We first prove that $\ell_{r_1},\ell_{r_2},\ldots$ is a sequence of consecutive records for $\imp$.
    Let $j\in\bN$ and consider $k<\ell_{r_i}$.
    Then, either we have $k=\ell_i$ with $i<r_1$ and then
    \begin{align*}
        \imp(k)=\mp{+}(i)<\mp{+}(r_j)=\imp(\ell_{r_j})\,,
    \end{align*}
    where the inequality uses that $r_j$ is a record index of $\mp{+}$;
    or we have $k=k_i$ for some $i\leq0$ and then
    \begin{align*}
        \imp(k)=1-\mp{-}(1-i)\leq0<\mp{+}(r_j)=\imp(\ell_{r_j})\,.
    \end{align*}
    In both case, $\imp(k)<\imp(\ell_{r_j})$ so $\ell_{r_j}$ is a record index of $\imp$.
    On the other hand, and using the same type of argument, if $\imp$ has a record index between $\ell_{r_i}$ and $\ell_{r_{i+1}}$, then there has to exist a record between $r_i$ and $r_{i+1}$ in $\mp{+}$, which contradicts the definition of $(r_1,r_2,\ldots)$.
    Before moving to the second part of the proof, we note that we always have $r_1=1$ and that $\imp(\ell_{r_1})>0$.

    \medskip

    We now prove that $\rho_1,\rho_2,\ldots$ are consecutive anti-record indices of $\mp{-}$, that is indices $\rho$ such that, for all $k>\rho$, $\mp{-}(\rho)<\mp{-}(k)$.
    The proof is divided into three steps: proving that this sequence is increasing, then proving that they are anti-record indices, and finally proving that there are not anti-record index between them.

    By their definition, we note that
    \begin{align*}
        \mp{-}(\rho_{i+1})=\min\big\{\mp{-}(k):k>\rho_i\big\}\,,
    \end{align*}
    and since $\mp{-}$ is bijective, this implies that $\rho_{i+1}>\rho_i$.
    Now, for any $i\in\bN$ and any $k>\rho_i$, we also have that $k\geq\rho_{i-1}$ and so, by definition
    \begin{align*}
        \mp{-}(\rho_i)\leq\mp{-}(k)\,.
    \end{align*}
    Using that $\mp{-}$ is a bijection along with the fact that $k\neq\rho_i$ leads to the desired strict inequality;
    this argument also works when $i=1$ since, by definition, for any $k\neq\mp{-}^{-1}(1)=\rho_1$, we have $\mp{-}(k)>1=\mp{-}(\rho_1)$.
    Finally, assume that there exists $\rho_i<k<\rho_{i+1}$ such that $k$ is an anti-record index of $\mp{-}$.
    Then, by definition of $k$ being an anti-record, we have that
    \begin{align*}
        \mp{-}(k)<\mp{-}(\rho_{i+1})\,.
    \end{align*}
    However, since $k>\rho_i$ and by definition of $\rho_{i+1}$, we also have that
    \begin{align*}
        \mp{-}(k)\geq\mp{-}(\rho_{i+1})\,,
    \end{align*}
    which is a contradiction.
    This proves that $\rho_1,\rho_2,\ldots$ is the ordered sequence of anti-record indices of $\mp{-}$

    \medskip

    We conclude the proof by proving that the desired sequence is the standard record representation of $\imp$.
    First of all, since $\imp(\ell_{r_1})>0$ and any other positive record of $\imp$ would also be a record of $\mp{+}$, the record prior to $\ell_{r_1}$ necessarily corresponds to the index $0$ in the standard record representation of $\imp$.
    Now, let $r$ be a record index of $\imp$ such that $\imp(r)\leq0$.
    Then $r=k_{1-i}$ for some $i\in\bN$ and so, for any $k>i$, by the definition of $r$ and since $k_{1-k}<k_{1-i}=r$, we have
    \begin{align*}
        \mp{-}(k)=1-\imp(k_{1-k})>1-\imp(k_{1-i})=\mp{-}(i)\,.
    \end{align*}
    This means that $i$ is an anti-record index of $\mp{-}$ and so any negative record of $\imp$ can be mapped to an anti-record of $\mp{-}$.
    For the same reason, the sequence of consecutive negative record of $\imp$ corresponds to the sequence of consecutive anti-records of $\mp{-}$.
    To conclude the proof, we just need to show that $k_{1-\rho_i}$ is not a record of $\imp$ for $i<m$.
    This holds since, by definition of $m$, for any $i<m$, we have $k_{1-\rho_i}>\ell_1=\ell_{r_1}$ (they cannot be equal) but also $\imp(k_{1-\rho_i})\leq0<\imp(\ell_{r_1})$.
\end{proof}

Proposition~\ref{prop:records imp} provides a first understanding of the records of $\imp$ given $(\mp{+},\mp{-},\Lambda)$.
However, the definition of $k_{1-\rho_m}$ might not be fully clear and we provide here a simpler characterization.

\begin{cor}\label{cor:record 0}
    Using the notations of Proposition~\ref{prop:records imp}, we have that
    \begin{align*}
        \imp(k_{1-\rho_m})=1-\min\Big\{\mp{-}(k):k>\Lambda_1\Big\}\,.
    \end{align*}
\end{cor}

\begin{proof}
    First recall that $m$ is defined so that $k_{1-\rho_m}<\ell_1<k_{1-\rho_{m-1}}$.
    Using that the sequence $\Lambda_1\geq\Lambda_2\geq\ldots$ is eventually equal to $0$, we observe that
    \begin{align*}
        \big|\big\{j\in\bZ\setminus\bN:k_j>\ell_1\big\}\big|=\Lambda_1\,.
    \end{align*}
    In particular, this implies that $k_{1-\Lambda_1-1}<\ell_1<k_{1-\Lambda_1}$.
    Since $m$ is defined so that $k_{1-\rho_m}<\ell_1<k_{1-\rho_{m-1}}$, it follows that $\rho_m>\Lambda_1\geq\rho_{m-1}$.

    Recall now the definition of $\rho_m$:
    \begin{align*}
        \rho_m=\mp{-}^{-1}\Big(\min\big\{\mp{-}(k):k>\rho_{m-1}\big\}\Big)\,,
    \end{align*}
    which equivalently means that
    \begin{align*}
        \mp{-}(\rho_m)=\min\big\{\mp{-}(k):k>\rho_{m-1}\big\}\,.
    \end{align*}
    Knowing that $\rho_m>\Lambda_1$, the minimum is not reached for $\rho_{m-1}<k\leq\Lambda_1$, so that
    \begin{align*}
        \mp{-}(\rho_m)=\min\big\{\mp{-}(k):k>\Lambda_1\big\}\,.
    \end{align*}
    To conclude the proof, recall that
    \begin{align*}
        \imp(k_{1-\rho_m})=1-\mp{-}(\rho_m)\,,
    \end{align*}
    giving us the desired result.
\end{proof}

Note that the statements of Proposition~\ref{prop:records imp} and Corollary~\ref{cor:record 0}, although applied to $\imp$, are actually deterministic results: they only depend on the construction of $\imp$ from $(\mp{+},\mp{-},\Lambda)$.
We now want to combine these results to characterize the distribution of the distance between consecutive records.
We first provide a characterization of the positive records of the two-sided infinite Mallows permutation.

\begin{cor}\label{cor:dist recs}
    Fix $q\in[0,1)$.
    Let $\imp$ be a two-sided infinite Mallows permutation and let $(\ldots,r_{-1},r_0,r_1,\ldots)$ be its standard record representation.
    Then the distance between consecutive positive records $(\imp(r_{i+1})-\max\{\imp(r_i),0\})_{i\geq0}$ is a sequence of independent $\pgeom{1-q}$ random variables.
\end{cor}

\begin{proof}
    Thanks to Proposition~\ref{prop:records imp}, we know that the positive records of $\imp$ are equal to the records of $\mp{+}$.
    But now~\cite[Proposition~1.11]{addario2021height} states that the distances between records is distributed as a $\pgeom{1-q}$ and that the consecutive values and number of records is a Markov chain, thus proving the independence of the sequence.
\end{proof}

While Corollary~\ref{cor:dist recs} characterizes the distribution of the jumps between positive records, proving a similar result for the negative proves to be much harder.
Indeed, negative records are related to anti-records of $\mp{-}$, and the sequence of anti-records of a Mallows permutation does not have independent geometric jumps.
However, their jump distribution is also biased by the shift created by $\Lambda$ and will thus end up being geometric, but it requires a different approach, which we now describe.

\medskip

In order to fully characterize the distribution of the distance between consecutive jumps in a two-sided infinite Mallows permutation, we shift the permutation so that negative records become related to records of $\mp{+}$.
More precisely, we strongly rely on the following proposition.

\begin{prop}\label{prop:shift rec}
    Fix $q\in[0,1)$.
    Let $\imp$ be the infinite Mallows permutation constructed from the triplet $(\mp{+},\mp{-},\Lambda)$, whose standard record representation is given by
    \begin{align*}
        (\ldots,k_{1-\rho_{m+2}},k_{1-\rho_{m+1}},\underline{k_{1-\rho_m}},\ell_{r_1},\ell_{r_2},\ldots)\,,
    \end{align*}
    as in Proposition~\ref{prop:records imp}.
    For any $n\in\bN$, further write $\imp^n$ for the permutation defined by $\imp^n(k)=\imp(n+k)-n$, for any $k\in\bZ$.
    Then, $\imp^n$ is distributed as a two-sided infinite Mallows permutation with parameter $q$.
    Moreover, by defining $s_n=\max\{i:\mp{+}(r_i)\leq n\}$, the standard record representation of $\imp^n$ is given by
    \begin{align*}
        (\ldots,k_{1-\rho_{m+1}}-n,k_{1-\rho_m}-n,\ell_{r_1}-n,\ldots,\ell_{r_{s_n-1}}-n,\underline{\ell_{r_{s_n}}-n},\ell_{r_{s_n+1}}-n,\ldots)\,.
    \end{align*}
\end{prop}

\begin{proof}
    This results mostly follows from results of~\cite{gnedin2012two}.
    Indeed, by letting $t^n(k)=k+n$ be the permutation translating indices by $n$, we have that $\imp^n=t^{-n}\cdot\imp\cdot t^n$.
    But then, thanks to~\cite[Proposition~2.3]{gnedin2012two} (the additive character of the balance) and~\cite[Theorem~3.3]{gnedin2012two} (the uniqueness of the two-sided Mallows distribution), $\imp^n$ is also a Mallows permutation.
    Finally, using the definition of records, we have that
    \begin{align*}
        \textrm{$r$ is a record of $\imp$}~&\Longleftrightarrow~\forall k<r:\imp(k)<\imp(r)\\
        &\Longleftrightarrow~\forall k<r-n:\imp(n+k)-n<\imp(n+(r-n))-n\\
        &\Longleftrightarrow~\forall k<r-n:\imp^n(k)<\imp^n(r-n)\\
        &\Longleftrightarrow~\textrm{$r-n$ is a record of $\imp^n$}\,,
    \end{align*}
    thus explaining the desired sequence of records.
    To conclude the proof, simply note that, by definition of $s_n$, we have $\mp{+}(r_{s_n})\leq n<\mp{+}(r_{s_n+1})$, leading to
    \begin{align*}
        \imp^n(\ell_{r_{s_n}}-n)&=\imp(\ell_{r_{s_n}})-n=\mp{+}(r_{s_n})-n\leq0
    \end{align*}
    and
    \begin{align*}
        \imp^n(\ell_{r_{s_n+1}}-n)&=\imp(\ell_{r_{s_n+1}})-n=\mp{+}(r_{s_n}+1)-n>0\,,
    \end{align*}
    thus proving the desired standard record representation.
\end{proof}

We now conclude this first part of the proof by providing a useful equality in distributions for the records of the two-sided infinite Mallows permutation.
The equality in distribution relies on the following transformation.
Let $g=(g_i)_{i\in\bZ}$ be a sequence of positive integers indexed by $\bZ$.
For any $n\geq0$, write $s_n=\min\{i\geq0:g_0+\ldots+g_i>n\}$ and let
\begin{align}\label{eq:def Phi}
    \Phi_n(g)&=\left(\ldots,g_{-2},g_{-1}+g_0-1,g_1,\ldots,g_{s_n-1},n+1-\sum_{i=0}^{s_n-1}g_i,\underline{\sum_{i=0}^{s_n}g_i-n},g_{s_n+1},\ldots\right)\,.
\end{align}
More precisely, $\Phi_n(g)=(\phi_i)_{i\in\bZ}$ is defined by
\begin{align*}
    \phi_i&=\left\{\begin{array}{ll}
        g_{s_n+i} & \textrm{if $i\geq1$} \\
        \sum_{i=0}^{s_n}g_i-n & \textrm{if $i=0$} \\
        n+1-\sum_{i=0}^{s_n-1}g_i & \textrm{if $i=-1$} \\
        g_{s_n+i+1} & \textrm{if $-s_n\leq i\leq-2$} \\
        g_0+g_{-1}-1 & \textrm{if $i=-s_n-1$} \\
        g_{s_n+i} & \textrm{if $i\leq-s_n-2$}\,,
    \end{array}\right.
\end{align*}
where we remove the fourth condition when $s_n=1$ and combine the third and fifth terms to obtain $n+g_{-1}$ when $s_n=0$.

\begin{cor}\label{cor:shift recs}
    Use the notations of Proposition~\ref{prop:shift rec} and the definition of $\Phi_n$ from~\eqref{eq:def Phi}.
    Write
    \begin{align*}
        &(\ldots,T_{-2},T_{-1},\underline{T_0},T_1,T_2,\ldots)\\
        &\pskip=(\ldots,\imp(k_{1-\rho_{m+2}})-1,\imp(k_{1-\rho_{m+1}})-1,\imp(k_{1-\rho_m})-1,\underline{0},\imp(\ell_{r_1}),\imp(\ell_{r_2}),\ldots)
    \end{align*}
    for the sequence of records of $\imp$ where we add $0$ at position $0$ and shift all negative records by $-1$.
    Then, for any $n\in\bN$, we have the equality in distribution
    \begin{align*}
        \Phi_n\big((T_{i+1}-T_i)_{i\in\bZ}\big)\disteq(T_{i+1}-T_i)_{i\in\bZ}\,.
    \end{align*}
\end{cor}

\begin{proof}
    This result directly follows from the equality in distribution stated in Proposition~\ref{prop:shift rec}.
    Indeed, the proposition implies that
    \begin{align}\label{eq:dist Ti}
        T_i&\disteq\left\{\begin{array}{l}
            \imp^n(\ell_{r_{s_n+i}}-n) \\
            0 \\
            \imp^n(\ell_{r_{s_n+i+1}}-n)-1 \\
            \imp^n(k_{1-\rho_{m-i-s_n-1}}-n)-1
        \end{array}\right.
        =\left\{\begin{array}{lll}
            T_{s_n+i}-n &~& \textrm{if $i\geq1$} \\
            0 &~& \textrm{if $i=0$} \\
            T_{s_n+i+1}-n-1 &~& \textrm{if $-s_n\leq i\leq -1$} \\
            T_{s_n+i}-n &~& \textrm{if $i\leq -s_n-1$}\,,
        \end{array}\right.
    \end{align}
    where the equality in distribution applies to the whole sequence of $(T_i)_{i\in\bZ}$.
    This then implies that
    \begin{align*}
        T_{i+1}-T_i&=\left\{\begin{array}{ll}
            T_{s_n+i+1}-T_{s_n+i} & \textrm{if $i\geq1$} \\
            T_{s_n+1}-n & \textrm{if $i=0$} \\
            n+1-T_{s_n} & \textrm{if $i=-1$} \\
            T_{s_n+i+2}-T_{s_n+i+1} & \textrm{if $-s_n\leq i\leq -2$} \\
            T_1-T_{-1}-1 & \textrm{if $i=-s_n-1$} \\
            T_{s_n+i+1}-T_{s_n+i} & \textrm{if $i\leq -s_n-2$}\,,
        \end{array}\right.
    \end{align*}
    and the right hand side is the exact definition of $\Phi_n$ (also in the case $s_n\in\{0,1\}$).
\end{proof}

In order to use this result, we now need to better understand what this equality in distribution means.
We explore properties of sequences satisfying the equality in distribution from Corollary~\ref{cor:shift recs} below.

\subsection{Geometric random variables and point processes}\label{sec:poisson}

The main goal of this section is to use the result of Corollary~\ref{cor:shift recs} to understand the global distribution of records of a two-sided infinite Mallows permutation.
In particular, we intend to prove the following theorem.

\begin{thm}\label{thm:poisson}
    Let $G=(G_i)_{i\in\bZ}$ be a sequence of random variables taking values in $\bZ$ and define $\Phi_n$ as in~\eqref{eq:def Phi}.
    Then, the two following properties are equivalent.
    \begin{itemize}
        \hypertarget{cond1}{\item[(P1)]}
        The full sequence $G=(G_i)_{i\in\bZ}$ is a family of independent $\pgeom{p}$ random variables.
        \hypertarget{cond2}{\item[(P2)]}
        The sub-sequence $(G_i)_{i\geq0}$ is a family of independent $\pgeom{p}$ random variables and the whole sequence $G=(G_i)_{i\in\bZ}$ satisfies the equality in distribution $\Phi_n(G)\disteq G$ for any $n\geq0$.
    \end{itemize}
\end{thm}

In order to prove Theorem~\ref{thm:poisson}, we use the natural link between geometric random variables and point processes.
In particular, we start by defining a point process which provides a sequence $(G_i)_{i\in\bZ}$ satisfying both~\cond{1} and \cond{2}.
We then use the existence of such sequence to prove that the two properties are equivalent.

\begin{proof}[Proof of Theorem~\ref{thm:poisson}]
    Fix $p\in(0,1]$.
    Let $P\subseteq\bZ$ be a random set of points obtained by independently keeping each $k\in\bZ$ with probability $p$.
    Define now the random process $(N_k)_{k\in\bZ}$ by $N_0=0$ and
    \begin{align*}
        N_k=\left\{\begin{array}{ll}
            \big|P\cap\llbracket0,k-1\rrbracket\big| & \textrm{if $k\geq1$} \\
            -\big|P\cap\llbracket k,-1\rrbracket\big| & \textrm{if $k\leq-1$}\,.
        \end{array}\right.
    \end{align*}
    Note that $(N_k)_{k\geq0}$ and $(-N_{-k})_{k\geq0}$ independent, identically distributed, and are both jumping processes with increments of size $1$ defined by
    \begin{align*}
        \prob(N_{k+1}=N_k)=\prob(N_{-k-1}=N_{-k})=1-p\,,
    \end{align*}
    for any $k\geq0$.
    This implies that the distance between jumping times are geometrically distributed and independent of each other.
    More precisely, by defining $T_0=0$ and, for any $i\geq1$, the stopping times by
    \begin{align*}
        T_i&=\min\{k\in\bN:N_k=i\}\\
        T_{-i}&=-\min\{k\in\bN:N_{-k}=-i\}\,,
    \end{align*}
    Then $(T_{i+1}-T_i)_{i\in\bZ}$ is a sequence of independent $\pgeom{p}$ random variables.
    Note that $T_i$ for $i\geq1$ is a standard stopping time for $(N_k)_{k\in\bZ}$, whereas $T_{-i}$ is not since it is defined by
    \begin{align*}
        T_{-i}=\max\{k:N_k=-i\}
    \end{align*}
    and so $\{T_{-i}\leq t\}$ belongs to the $\sigma$-algebra generated by $\{N_k:k\leq t+1\}$ but not that of $\{N_k:k\leq t\}$.
    We now use this point process and the related variables to prove the statement of Theorem~\ref{thm:poisson}.
    
    \medskip
    
    First, let $G_i=T_{i+1}-T_i$ so that $(G_i)_{i\in\bZ}$ satisfies~\cond{1}.
    We now prove that it further satisfies~\cond{2}.
    Note that $G_0+\ldots+G_i=T_{i+1}$.
    Moreover using the definition of $s_n$ and the fact that that $\{T_i>k\}=\{N_k<i\}$ for $i,k\geq0$, we obtain
    \begin{align*}
        s_n&=\min\Big\{i\in\bN:T_{i+1}>n\Big\}=\min\Big\{i\in\bN:N_n<i+1\Big\}=N_n\,.
    \end{align*}
    
    Consider now the new process defined by $N^n=(N^n_k)_{k\in\bZ}=(N_{n+k}-N_n)_{k\in\bZ}$.
    First of all, note that $N^n_0=0$.
    Moreover, for any $k\geq1$, we have
    \begin{align*}
        N^n_k&=\big|P\cap\llbracket0,n+k-1\rrbracket\big|-\big|P\cap\llbracket0,n-1\rrbracket\big|=\big|(n+P)\cap\llbracket0,k-1\rrbracket\big|\,.
    \end{align*}
    Similarly, for any $-n\leq k\leq -1$, we have
    \begin{align*}
        N^n_k&=\big|P\cap\llbracket0,n-|k|-1\rrbracket\big|-\big|P\cap\llbracket0,n-1\rrbracket\big|=-\big|(n+P)\cap\llbracket-|k|,-1\rrbracket\big|=-\big|(n+P)\cap\llbracket k,-1\rrbracket\big|\,.
    \end{align*}
    Finally, for $k\leq-n-1$, we have
    \begin{align*}
        N^n_k&=-\big|P\cap\llbracket n+k,-1\rrbracket\big|-\big|P\cap\llbracket0,n-1\rrbracket\big|=-\big|(n+P)\cap\llbracket k,-1\rrbracket\big|\,.
    \end{align*}
    This means that $N^n$ is defined in the same way as $N$ but with $P^n=n+P$ instead of $P$.
    By translation invariance of the distribution of $P$, it follows that $N^n$ is distributed as $N$.
    
    To conclude the proof that $(G_i)_{i\in\bZ}$ satisfies~\cond{2}, define $(T^n_i)_{i\in\bZ}$ for $N^n$ in the same way that $(T_i)_{i\in\bZ}$ is defined with regards to $N$.
    Then, we verify that
    \begin{align*}
        T^n_i=\left\{\begin{array}{ll}
            T_{i+N_n}-n & \textrm{if $i\geq1$} \\
            0 & \textrm{if $i=0$} \\
            T_{i+N_n+1}-n-1 & \textrm{if $-N_n\leq i\leq-1$} \\
            T_{i+N_n}-n & \textrm{if $i\leq-N_n-1$}\,.
        \end{array}\right.
    \end{align*}
    But this set of equality is exactly the same as in~\eqref{eq:dist Ti} from the proof of Corollary~\ref{cor:shift recs}.
    This proves that $(G_i)_{i\in\bZ}$ satisfies~\cond{2}.
    
    \medskip
    
    To conclude the proof of Theorem~\ref{thm:poisson}, first note that~\cond{1} fully characterizes the distribution of $(G_i)_{i\in\bZ}$ and thus, thanks to the fact that our previous construction of $(G_i)_{i\in\bZ}$ also satisfies~\cond{2}, we have proven that $\cond{1}\Rightarrow\cond{2}$.
    We now prove the converse result by showing that the property~\cond{2} fully characterizes the distribution of $(G_i)_{i\in\bZ}$.
    To avoid any confusion, we now write $(\Tilde{G}_i)_{i\in\bZ}$ for a general sequence of random variables satisfying~\cond{2} and keep $(G_i)_{i\in\bZ}$ for a sequence of independent $\pgeom{p}$ (which thus satisfies~\cond{1} and \cond{2}).
    We further extend the notations to define $\Tilde{s}_n=\min\{i\geq0:\Tilde{G}_0+\ldots+\Tilde{G}_i>n\}$.
    
    Let $k\in\bZ$ be a fixed integer and $g_{-k},\ldots,g_k$ be a sequence of $2k+1$ integers.
    We want to prove that
    \begin{align*}
        \prob\big(\Tilde{G}_i=g_i,\forall|i|\leq k\big)=\prob\big(G_i=g_i,\forall|i|\leq k\big)\,.
    \end{align*}
    To do so, write $\Phi_n^k(G)=(\phi_i)_{|i|\leq k}$ for the $2k+1$ terms of the sequence $\Phi_n(G)=(\phi_i)_{i\in\bZ}$ (defined in~\eqref{eq:def Phi}) around the central term.
    In particular, note that $\Phi_n^k(G)$ belongs to the $\sigma$-algebra generated by $\{(G_i)_{i\geq0},(G_i)_{s_n-k\leq i\leq0}\}$ and that $s_n$ belongs to the $\sigma$-algebra generated by $(G_i)_{i\geq0}$.
    We define in a similar way $\Phi_n^k(\Tilde{G})$.
    
    Using the fact that $(\Tilde{G}_i)_{i\in\bZ}$ satisfies~\cond{2}, for any $n\in\bN$, we have that
    \begin{align*}
        \prob\big(\Tilde{G}_i=g_i,\forall|i|\leq k\big)=\prob\big(\Phi_n^k(\Tilde{G})=(g_i)_{|i|\leq k}\big)\,.
    \end{align*}
    Now, due to the heavy tail of geometric random variables, we know that $\Tilde{s}_n$ almost-surely diverges to infinity as $n$ goes to infinity.
    This implies that, as $n$ goes to infinity, we have
    \begin{align*}
        \prob\big(\Tilde{G}_i=g_i,\forall|i|\leq k\big)=\prob\big(\Phi_n^k(\Tilde{G})=(g_i)_{|i|\leq k},\Tilde{s}_n\geq k\big)+o(1)\,.
    \end{align*}
    But now, thanks to the equality in distribution, we know that
    \begin{align*}
        \prob\big(\Phi_n^k(\Tilde{G})=(g_i)_{|i|\leq k},\Tilde{s}_n\geq k\big)=\prob\big(\Phi_n^k(G)=(g_i)_{|i|\leq k},s_n\geq k\big)\,.
    \end{align*}
    Finally, using the same argument as that of $(\Tilde{G}_i)_{i\in\bZ}$ but on $(G_i)_{i\in\bZ}$ instead, the previous equations imply that
    \begin{align*}
        \prob\big(\Tilde{G}_i=g_i,\forall|i|\leq k\big)&=\prob\big(\Phi_n^k(\Tilde{G})=(g_i)_{|i|\leq k},\Tilde{s}_n\geq k\big)+o(1)\\
        &=\prob\big(\Phi_n^k(G)=(g_i)_{|i|\leq k},s_n\geq k\big)+o(1)\\
        &=\prob\big(G_i=g_i,\forall|i|\leq k\big)+o(1)\,.
    \end{align*}
    Since the two probabilities do not depend on $n$, they are equal.
    This proves that $(G_i)_{i\in\bZ}$ and $(\Tilde{G}_i)_{i\in\bZ}$ have the same distribution and concludes the proof of Theorem~\ref{thm:poisson}.
\end{proof}

\subsection{Distribution of the two-sided infinite Mallows tree}\label{sec:imt}

In this section, we finally prove Theorem~\ref{thm:imt} using the results of Sections~\ref{sec:rec imt} and \ref{sec:poisson} above.

\begin{proof}[Proof of Theorem~\ref{thm:imt}]
    As mentioned in the introduction, the Mallows distribution (on any subset of $\bZ$) is the only $q$-invariant distribution on permutations.
    Thus, for any finite subset $I\subseteq\bZ$, by letting $\imp\langle I\rangle$ be the unique permutation on $[|I|]=\{1,\ldots,|I|\}$ whose images are in the same order as that of $\imp$ on $\imp^{-1}(I)$, it follows that $\imp\langle I\rangle$ is distributed as $\pi_{|I|,q}$ and is moreover independent of $\imp$ on $\bZ\setminus\imp^{-1}(I)$.
    Let now $(R_i)_{i\in\bZ}$ be the sequence of records whose indices coincides with that of the standard record representation of $\imp$ and write $\imp[i]=\imp\langle\llbracket R_i+1,R_{i+1}-1\rrbracket\rangle$ for the induced permutations between consecutive records.
    Then, using the previous property of $\imp$, conditionally given $(R_i)_{i\in\bZ}$, the sequence $(\imp^i)_{i\in\bZ}$ is a family of independent Mallows permutations with parameter $q$, and their sizes are defined by the distances between consecutive records, but do not depend on the values of the records.
    Finally, recall the definition of redwood trees in~\eqref{eq:redwood} as
    \begin{align*}
        Y_q\overset{d}{=}\ribst{\imp}=\bigcup_{k\in\bZ}\{\rv^k\}\cup\rv^k\lv\bst{\imp[k]}
    \end{align*}
    to see that it only remains to prove that the statement regarding the distribution of the sizes of the trees $(|\imp(k)|)_{k\in\bZ}$ from Theorem~\ref{thm:imt} is true.
    Indeed, the trees $(\imt(k))_{k\in\bZ}$ exactly correspond to $(\rv^k\lv\bst{\imp[k]})_{k\in\bZ}$ and are thus, conditionally given their sizes, independent Mallows trees with parameter $q$.

    \medskip

    Consider the distribution of $(|\imt(k)|)_{k\in\bZ}$, which corresponds to that of $(R_{k+1}-R_k-1)_{k\in\bZ}$.
    First of all, define $(T_i)_{i\in\bZ}$ as in Corollary~\ref{cor:shift recs}, that is
    \begin{align*}
        T_i=\left\{\begin{array}{ll}
            R_i & \textrm{if $i\geq1$} \\
            0 & \textrm{if $i=0$} \\
            R_{i+1}-1 & \textrm{if $i\leq-1$}\,.
        \end{array}\right.
    \end{align*}
    Further let $G=(G_i)_{i\in\bZ}$ be defined by $G_i=T_{i+1}-T_i$ for all $i\in\bZ$.
    Thanks to Corollary~\ref{cor:shift recs}, we know that $\Phi_n(G)=G$ for all $n\geq0$, where $\Phi_n$ is defined in~\eqref{eq:def Phi}.
    Moreover, Corollary~\ref{cor:dist recs} states that $(G_i)_{i\geq0}$ is a sequence of independent $\pgeom{1-q}$, so $G$ satisfies~\cond{2} of Theorem~\ref{thm:poisson} and thus also satisfies~\cond{1}: $G=(G_i)_{i\in\bZ}$ is a sequence of independent $\pgeom{1-q}$ random variables.

    To conclude the proof of the theorem, observe that
    \begin{align*}
        R_{i+1}-R_i-1=\left\{\begin{array}{ll}
            G_i-1 & \textrm{if $i\geq1$} \\
            G_0+G_{-1}-2 & \textrm{if $i=0$} \\
            G_{i-1}-1 & \textrm{if $i\leq -1$}
        \end{array}\right.
    \end{align*}
    to first see that the sizes of the subtrees $(\imt(k))_{k\in\bZ}$ are all independent and their distribution is $\geom{1-q}$ when $k\neq0$.
    Finally, the distribution of $\imt(0)$ is equal to the sum of two independent $\geom{1-q}$ and is thus given by
    \begin{align}\label{eq:geomsing}
        \prob(\imt(0)=n)&=\prob(G_0+G_{-1}=2+n)=\sum_{g=1}^{1+n}q^{g-1}(1-q)q^{2+n-g-1}(1-q)=(n+1)q^n(1-q)^2\,,
    \end{align}
    which happens to be exactly that of $\sbgeom{1-q}$ as defined in~\eqref{eq:all geoms}, thus concluding the proof of the theorem.
\end{proof}

\subsection{Local limit of Mallows trees}\label{sec:local limit}

We conclude this section with the proof of Theorem~\ref{thm:conv}, showing that finite Mallows trees almost-surely locally converge to the redwood tree of the two-sided infinite Mallows permutation.
We separate the proof into three sub-propositions and an intermediate lemma.

\medskip

For the rest of this section we fix an integer $r\geq1$ and a rooted graph $H$.
We also assume without loss of generality that $(\mp{n})_{n\in\bN\cup\{\infty\}}$ is coupled according to the infinite Bernoulli model, as described in Section~\ref{sec:finite mallows}, and we write $\mt{n}=\bst{\mp{n}}$ for $n\in\bN\cup\{\infty\}$.
The goal of this section is thus to prove that $\mt{n}$ satisfies~\eqref{eq:las} with $\imt$, that is
\begin{align}\label{eq:target}
    \frac{1}{n}\sum_{v\in\mt{n}}\mathbbm{1}_{\{B_r(v,\mt{n})\equiv H\}}\asconv\prob\big(B_r(\imt)\equiv H\big)\,.
\end{align}
We also recall that $\imt$ is defined as $(o,Y_q)$, where $Y_q=\ribst{\imp}$ is the redwood tree of a two-sided infinite Mallows permutation $\imp$ and $o$ is a uniformly sampled vertex of $Y_q\cap\leftT$, where $\leftT=\{\root\}\cup\lv\BST$.
By writing $Y_q(k)=Y_q\cap\rv^k\lv\BST$ as in Theorem~\ref{thm:imt}, this implies that
\begin{align}\label{eq:imt dist}
    \prob\big(B_r(\imt)\equiv H\big)&=\prob\left(B_r\left(o,\bigcup_{|k|\leq r}\{\rv^k\}\cup Y_q(k)\right)\equiv H\right)\,,
\end{align}
which will prove to be a useful characterization of the target limit.
We also extend the notation $Y_q(k)$ to $\mt{\infty}$ by letting $\mt{\infty}(k)=\mt{\infty}\cap\rv^k\lv\BST$ for any $k\geq0$.
We mention here that the sequence $(\mt{\infty}(k))_{k\geq0}$ is known to be a sequence of independent Mallows trees with $\geom{1-q}$ sizes (see~\cite[Lemma~1.10]{addario2021height}).

We now state and prove the first proposition, which allows us to changed the left-hand side in~\eqref{eq:target} by an object with less inter-dependency.

\begin{prop}\label{prop:new target}
    Using the notations introduced at the beginning of Section~\ref{sec:local limit}, we have that
    \begin{align*}
        \frac{1}{n}\sum_{v\in\mt{n}}\mathbbm{1}_{\{B_r(v,\mt{n})\equiv H\}}-\frac{1}{n}\sum_{k=0}^{\lfloor n(1-q)\rfloor}\sum_{v\in\{\rv^k\}\cup\mt{\infty}(k)}\mathbbm{1}_{\{B_r(v,\mt{\infty})\equiv H\}}\asconv0\,.
    \end{align*}
\end{prop}

\begin{proof}
    Recall the definition of $\reco{n}$ from~\eqref{eq:records} and $\free{n}$ from~\eqref{eq:spaces}, leading to~\eqref{eq:bound left}:
    \begin{align*}
        \Big|\Big\{0\leq k\leq\reco{n}:\bst{\mp{\infty}}\cap\rv^k\leftT\neq\bst{\mp{n}}\cap\rv^k\leftT\Big\}\Big|\leq\free{n}\,.
    \end{align*}
    Further let
    \begin{align*}
        \mt{\infty}^n=\mt{\infty}\cap\bigcup_{0\leq k\leq\reco{n}}\rv^k\leftT
    \end{align*}
    be the infinite tree $\mt{\infty}$ cropped at depth $\reco{n}$ on the rightmost path, or equivalently the tree $\mt{n}$ where we fill in the left subtrees as they appear in their limit $\mt{\infty}$.
    We recall here that, by definition, $\free{n}=|\mt{\infty}^n\setminus\mt{n}|$.
    
    For any binary search tree $T$ and $v\in T\cap\rv^k\leftT$, we see that $B_r(v,T)$ has size at most $3\cdot2^r$ and so we obtain
    \begin{align*}
        \left|\sum_{v\in\mt{n}}\mathbbm{1}_{\{B_r(v,\mt{n})\equiv H\}}-\sum_{v\in\mt{n}}\mathbbm{1}_{\{B_r(v,\mt{\infty})\equiv H\}}\right|\leq(3\cdot2^r)\free{n}\,.
    \end{align*}
    Using the definition of $\free{n}$, we also see that
    \begin{align*}
        &\sum_{v\in\mt{\infty}^n}\mathbbm{1}_{\{B_r(v,\mt{\infty})\equiv H\}}-\sum_{v\in\mt{n}}\mathbbm{1}_{\{B_r(v,\mt{\infty})\equiv H\}}=\sum_{v\in\mt{\infty}^n\setminus\mt{n}}\mathbbm{1}_{\{B_r(v,\mt{\infty})\equiv H\}}\leq\free{n}\,,
    \end{align*}
    Combining the previous results leads to
    \begin{align*}
        &\left|\frac{1}{n}\sum_{v\in\mt{n}}\mathbbm{1}_{\{B_r(v,\mt{n})\equiv H\}}-\frac{1}{n}\sum_{v\in\mt{\infty}^n}\mathbbm{1}_{\{B_r(v,\mt{\infty})\equiv H\}}\right|\\
        &\pskip\leq\left|\frac{1}{n}\sum_{v\in\mt{n}}\mathbbm{1}_{\{B_r(v,\mt{n})\equiv H\}}-\frac{1}{n}\sum_{v\in\mt{n}}\mathbbm{1}_{\{B_r(v,\mt{\infty})\equiv H\}}\right|+\frac{1}{n}\sum_{v\in\mt{\infty}^n\setminus\mt{n}}\mathbbm{1}_{\{B_r(v,\mt{\infty})\equiv H\}}\\
        &\pskip\leq\frac{(3\cdot2^r+1)\free{n}}{n}\,,
    \end{align*}
    But now, thanks to Lemma~\ref{lem:spaces} and since $r$ is fixed, this implies that
    \begin{align}\label{eq:Tn and Tinfty}
        \frac{1}{n}\sum_{v\in\mt{n}}\mathbbm{1}_{\{B_r(v,\mt{n})\equiv H\}}-\frac{1}{n}\sum_{v\in\mt{\infty}^n}\mathbbm{1}_{\{B_r(v,\mt{\infty})\equiv H\}}\asconv0\,.
    \end{align}
    Then, using the definition of $\mt{\infty}^n$, we observe that
    \begin{align*}
        \sum_{v\in\mt{\infty}^n}\mathbbm{1}_{\{B_r(v,\mt{\infty})\equiv H\}}=\sum_{k=0}^{\reco{n}}\sum_{v\in\{\rv^k\}\cup\mt{\infty}(k)}\mathbbm{1}_{\{B_r(v,\mt{\infty})\equiv H\}}\,.
    \end{align*}
    which is almost the desired result, except that we need to replace $\reco{n}$ with $n(1-q)$.

    \medskip
    
    Using Lemma~\ref{lem:records}, we know that, almost-surely, $|\reco{n}-\lfloor n(1-q)\rfloor|\leq\sqrt{n}\log n$ for $n$ large enough.
    Moreover, using the strong law of large number along with the fact that $(\mt{\infty}(k))_{k\geq0}$ are independent Mallows trees of geometric sizes, we have that
    \begin{align*}
        \frac{1}{2\sqrt{n}\log n}\sum_{k:|k-\lfloor n(1-q)\rfloor|\leq\sqrt{n}\log n}\Big(1+\big|\mt{\infty}(k)\big|\Big)\asconv\frac{1}{1-q}\,,
    \end{align*}
    where the limit is simply the expected value of a $\pgeom{1-q}$ random variable.
    Now, for $n$ large enough so that $|\reco{n}-\lfloor n(1-q)\rfloor|\leq\sqrt{n}\log n$ almost-surely, we have that
    \begin{align*}
        &\left|\frac{1}{n}\sum_{v\in\mt{\infty}^n}\mathbbm{1}_{\{B_r(v,\mt{\infty})\equiv H\}}-\frac{1}{n}\sum_{k=0}^{\lfloor n(1-q)\rfloor}\sum_{v\in\{\rv^k\}\cup\mt{\infty}(k)}\mathbbm{1}_{\{B_r(v,\mt{\infty})\equiv H\}}\right|\\
        &\pskip=\frac{1}{n}\left|\sum_{k=0}^{\reco{n}}\sum_{v\in\{\rv^k\}\cup\mt{\infty}(k)}\mathbbm{1}_{\{B_r(v,\mt{\infty})\equiv H\}}-\sum_{k=0}^{\lfloor n(1-q)\rfloor}\sum_{v\in\{\rv^k\}\cup\mt{\infty}(k)}\mathbbm{1}_{\{B_r(v,\mt{\infty})\equiv H\}}\right|\\
        &\pskip\leq\frac{2\sqrt{n}\log n}{n}\cdot\left[\frac{1}{2\sqrt{n}\log n}\sum_{k:|k-\lfloor n(1-q)\rfloor|\leq\sqrt{n}\log n}\Big(1+\big|\mt{\infty}(k)\big|\Big)\right]\,.
    \end{align*}
    If follows that
    \begin{align*}
        \frac{1}{n}\sum_{v\in\mt{\infty}^n}\mathbbm{1}_{\{B_r(v,\mt{\infty})\equiv H\}}-\frac{1}{n}\sum_{k=0}^{\lfloor n(1-q)\rfloor}\sum_{v\in\{\rv^k\}\cup\mt{\infty}(k)}\mathbbm{1}_{\{B_r(v,\mt{\infty})\equiv H\}}\asconv0\,,
    \end{align*}
    which combined with~\eqref{eq:Tn and Tinfty} provides us with the desired almost-sure convergence.
\end{proof}

The previous proposition provides an equivalent sequence of random variables to consider in order to prove that~\eqref{eq:target} holds.
We now state and prove the second proposition, showing the existence and a first characterization of the limit of the desired sequence of random variables.

\begin{prop}\label{prop:a limit}
    Using the notations introduced at the beginning of Section~\ref{sec:local limit}, we have that
    \begin{align*}
        \frac{1}{n}\sum_{k=0}^{\lfloor n(1-q)\rfloor}\sum_{v\in\{\rv^k\}\cup\mt{\infty}(k)}\mathbbm{1}_{\{B_r(v,\mt{\infty})\equiv H\}}\asconv(1-q)L\,,
    \end{align*}
    where
    \begin{align}\label{eq:def L}
        L=\expec\left[\sum_{v\in\{\rv^r\}\cup\mt{\infty}(r)}\mathbbm{1}_{\{B_r(v,\mt{\infty})\equiv H\}}\right]\,.
    \end{align}
\end{prop}

\begin{proof}
    Consider some $k\geq0$ and let $v\in\{\rv^k\}\cup\mt{\infty}(k)$.
    Since $B_r(v,\mt{\infty})$ is the ball of radius $r$ around $v$, it only depends on $(\mt{\infty}(\ell))_{|\ell-k|\leq r}$.
    Moreover, if $k\geq r$, then $(\mt{\infty}(\ell))_{|\ell-k|\geq r}$ is distributed as $(\mt{\infty}(\ell))_{0\leq\ell\leq 2r}$  and thus, the family of random variables
    \begin{align}\label{eq:weird family}
        \left(\sum_{v\in\{\rv^k\}\cup\mt{\infty}(k)}\mathbbm{1}_{\{B_r(v,\mt{\infty})\equiv H\}}\right)_{k\geq r}
    \end{align}
    are identically distributed and their expectation is given by $L$ as in~\eqref{eq:def L}.
    They are however not independent (since $r\geq1$), but this is actually only a minor problem, since their dependency only has finite range.
    Indeed, for any $k$ and $k'$ such that $|k-k'|\geq2r$, the random variables of~\eqref{eq:weird family} with indices $k$ and $k'$ are independent.
    Note that, in the case $|k-k'|=2r$, of all the pairs of balls of radius $r$ considered in the two random variables, exactly one of them intersects, when $v=\rv^k$ and $v'=\rv^{k'}$, and the intersection is $B_r(\rv^k,\mt{\infty})\cap B_r(\rv^{k'},\mt{\infty})=\{\rv^{(k+k')/2}\}$;
    however, since this intersection is deterministic, the two random variables are still independent.
    
    Consider now any $r\leq k<3r$.
    Thanks to the previous observation of independence and identical distribution, by applying the strong law of large numbers, we know that
    \begin{align*}
        \frac{1}{1+\left\lfloor\frac{n(1-q)-k}{2r}\right\rfloor}\sum_{\{i:k+2ri\leq n(1-q)\}}\sum_{\{v\in\{\rv^{k+2ri}\}\cup\mt{\infty}(k+2ri)\}}\mathbbm{1}_{\{B_r(v,\mt{\infty})\equiv H\}}\asconv L\,,
    \end{align*}
    Since this is true for any $r\leq k<3r$ and $r$ is fixed and finite, it follows that the sum of all the previous random variables converges almost-surely to $2rL$.
    Using that $\frac{1}{1+\left\lfloor\frac{n(1-q)-k}{2r}\right\rfloor}\sim\frac{2r}{n(1-q)}$, this implies that
    \begin{align*}
        \sum_{k=r}^{3r-1}\frac{2r}{n(1-q)}\sum_{\{i:k+2ri\leq n(1-q)\}}\sum_{\{v\in\{\rv^{k+2ri}\}\cup\mt{\infty}(k+2ri)\}}\mathbbm{1}_{\{B_r(v,\mt{\infty})\equiv H\}}\asconv2rL
    \end{align*}
    and so
    \begin{align*}
        \frac{1}{n}\sum_{k=r}^{\lfloor n(1-q)\rfloor}\sum_{v\in\{\rv^k\}\cup\mt{\infty}(k)}\mathbbm{1}_{\{B_r(v,\mt{\infty})\equiv H\}}\asconv(1-q)L\,.
    \end{align*}
    To conclude the proof of the proposition, use that the first $r$ subtrees of $\mt{\infty}$ are almost-surely finite to extend the previous sum to $0\leq k\leq n(1-q)$ and obtain the desired almost-sure convergence.
\end{proof}

With this second proposition, we obtain that the equivalent sequence of random variables from Proposition~\ref{prop:new target} admits a limit.
It now only remains to prove that the limit is as desired.
Before doing so, we provide a lemma which will be useful in proving the final proposition.

\begin{lemma}\label{lem:size biased}
    Let $(S_i)_{i\geq0}$ be a sequence of disjoint sets such that $(|S_i|)_{i\geq0}$ is a sequence of independent random variables distributed according to $\pi=(\pi_\ell)_{\ell\geq1}$.
    Write $\mu=\sum_{\ell\geq1}\ell\pi_\ell$ for the mean of $\pi$.
    Fix an integer $r\geq1$ and a sequence of integers $(k_n)_{n\in\bN}$ such that $k_n\rightarrow\infty$ as $n$ goes to infinity.
    Finally, let $U_n$ be a uniform random variable on the set $\bigcup_{i=0}^{k_n}S_i$.
    Then, for any $\ell_0,\ldots,\ell_r$, we have
    \begin{align*}
        \lim_{n\rightarrow\infty}\prob\Big(|S_i|=\ell_i,\forall i\leq r~\Big|~U_n\in S_0\Big)=\frac{\ell_0\pi_{\ell_0}}{\mu}\prod_{i=1}^r\pi_{\ell_i}\,.
    \end{align*}
    In words, conditionally given that $U_n\in S_0$, the random variables $(S_0,\ldots,S_r)$ are independent with $S_0$ being distributed as the size-biased version of $\pi$ and the other $S_i$ being distributed according to $\pi$.
\end{lemma}

\begin{proof}
    Using the rules of conditional probability, we know that
    \begin{align*}
        \prob\Big(|S_i|=\ell_i,\forall i\leq r~\Big|~U_n\in S_0\Big)&=\frac{\prob\Big(|S_i|=\ell_i,\forall i\leq r\Big)}{\prob\Big(U_n\in S_0\Big)}\prob\Big(U_n\in S_0~\Big|~|S_i|=\ell_i,\forall i\leq r\Big)\,.
    \end{align*}
    Since the $(S_i)_{0\leq i\leq k_n}$ are translation invariant, we know that
    \begin{align*}
        \prob\Big(U_n\in S_0\Big)=\frac{1}{k_n+1}\,,
    \end{align*}
    which implies that
    \begin{align*}
        \prob\Big(|S_i|=\ell_i,\forall i\leq r~\Big|~U_n\in S_0\Big)&=\left((k_n+1)\prod_{i=0}^r\pi_{\ell_i}\right)\cdot\prob\Big(U_n\in S_0~\Big|~|S_i|=\ell_i,\forall i\leq r\Big)\,.
    \end{align*}
    Now, conditioning on the values of the other $S_i$ (assuming that $n$ is large enough so that $k_n\geq r$), we have that
    \begin{align*}
        \prob\Big(U_n\in S_0~\Big|~|S_i|=\ell_i,\forall i\leq r\Big)&=\expec\bigg[\prob\Big(U_n\in S_0~\Big|~|S_i|=\ell_i,\forall i\leq r,\big(|S_i|\big)_{r<i\leq k_n}\Big)\bigg]\\
        &=\expec\left[\frac{\ell_0}{\sum_{i=0}^r\ell_i+\sum_{i=r+1}^{k_n}|S_i|}\right]\,,
    \end{align*}
    which, once plugged back into the previous equality, leads to
    \begin{align*}
        \prob\Big(|S_i|=\ell_i,\forall i\leq r~\Big|~U_n\in S_0\Big)&=\ell_0\prod_{i=0}^r\pi_{\ell_i}\cdot\expec\left[\frac{k_n+1}{\sum_{i=0}^r\ell_i+\sum_{i=r+1}^{k_n}|S_i|}\right]\,.
    \end{align*}
    Finally, using the fact that $k_n\rightarrow\infty$ along with the strong law of large numbers and the dominated convergence theorem (using that $|S_i|\geq1$), we obtain that
    \begin{align*}
        \lim_{n\rightarrow\infty}\expec\left[\frac{k_n+1}{\sum_{i=0}^r\ell_i+\sum_{i=r+1}^{k_n}|S_i|}\right]=\frac{1}{\mu}\,,
    \end{align*}
    which provides the desired limit.
\end{proof}

We finally prove the third proposition used to prove Theorem~\ref{thm:conv}, which characterizes the limit obtained in Proposition~\ref{prop:a limit}.

\begin{prop}\label{prop:correct lim}
    Using the notations introduced at the beginning of the section and $L$ defined in~\eqref{eq:def L}, we have that
    \begin{align*}
        (1-q)L=\prob\big(B_r(\imt)\equiv H\big)\,.
    \end{align*}
\end{prop}

\begin{proof}
    Using the definition of $L$ from~\eqref{eq:def L} and its relation with the family of random variables given in~\eqref{eq:weird family}, we have that
    \begin{align*}
        (1-q)L&=(1-q)\expec\left[\sum_{v\in\{\rv^r\}\cup\mt{\infty}(r)}\mathbbm{1}_{\{B_r(v,\mt{\infty})\equiv H\}}\right]\\
        &=(1-q)\expec\left[\frac{1}{\lfloor n(1-q)\rfloor+1-r}\sum_{k=r}^{\lfloor n(1-q)\rfloor}\sum_{v\in\{\rv^k\}\cup\mt{\infty}(k)}\mathbbm{1}_{\{B_r(v,\mt{\infty})\equiv H\}}\right]\,.
    \end{align*}
    For more clarity in the rest of this proof, let
    \begin{align*}
        \lastT=\bigcup_{k=r}^{\lfloor n(1-q)\rfloor}\{\rv^k\}\cup\mt{\infty}(k)
    \end{align*}
    be the tree which will help us characterize $(1-q)L$;
    in particular, $\lastT$ corresponds to $\mt{\infty}^n$ defined in the proof of Proposition~\ref{prop:new target} where we remove the first $r$ left subtrees.
    Then, combining the two previous equations, we have
    \begin{align*}
        (1-q)L&=(1-q)\expec\left[\frac{1}{\lfloor n(1-q)\rfloor+1-r}\sum_{v\in\lastT}\mathbbm{1}_{\{B_r(v,\mt{\infty})\equiv H\}}\right]\\
        &=\frac{n(1-q)}{\lfloor n(1-q)\rfloor+1-r}\expec\left[\frac{|\lastT|}{n}\cdot\frac{1}{|\lastT|}\sum_{v\in\lastT}\mathbbm{1}_{\{B_r(v,\mt{\infty})\equiv H\}}\right]\,.
    \end{align*}
    We note that the term in front of the expectation converges to $1$ as $n$ goes to infinity.
    Moreover, $|\lastT|$ is a sum of $\lfloor n(1-q)\rfloor-r+1$ independent $\pgeom{1-q}$ random variables and so it is tight and satisfies the strong law of large numbers: $|\lastT|/n\asconv1$.
    Finally, note that
    \begin{align*}
        \expec\left[\frac{1}{|\lastT|}\sum_{v\in\lastT}\mathbbm{1}_{\{B_r(v,\mt{\infty})\equiv H\}}\right]&=\prob\Big(B_r(V_n,\mt{\infty})\equiv H\Big)\,,
    \end{align*}
    where $V_n$ is a uniformly sampled vertex from $\lastT$.
    So now, if we manage the prove that the right-hand side of the previous equation admits a limit as $n$ goes to infinity, by the dominated convergence theorem we can conclude that
    \begin{align*}
        (1-q)L&=\lim_{n\rightarrow\infty}\prob\Big(B_r(V_n,\mt{\infty})\equiv H\Big)\,.
    \end{align*}
    We now show that the limit indeed exists and corresponds to the desired value.

    \medskip
    
    Using that the subtrees $(\mt{\infty}(k))_{k\geq0}$ are independent and identically distributed, we have
    \begin{align*}
        \prob\Big(B_r(V_n,\mt{\infty})\equiv H\Big)&=\sum_{k=r}^{\lfloor n(1-q)\rfloor}\prob\Big(B_r(V_n,\mt{\infty})\equiv H~\Big|~V_n\in\lastT\cap\rv^k\leftT\Big)\prob\Big(V_n\in\lastT\cap\rv^k\leftT)\Big)\\
        &=\prob\Big(B_r(V_n,\mt{\infty})\equiv H~\Big|~V_n\in\lastT\cap\rv^r\leftT\Big)\\
        &=\prob\Big(B_r(V_n,\mt{\infty})\equiv H~\Big|~V_n\in\{\rv^k\}\cup\mt{\infty}(\rv^r)\Big)\,.
    \end{align*}
    The event $V_n\in\{\rv^k\}\cup\mt{\infty}(\rv^r)$ only depends on the size of $|\mt{\infty}(\rv^r)|$ in comparison to the other left subtrees $(|\mt{\infty}(k)|)_{r<k\leq n(1-q)}$ and is independent of their structure as well as that of the other left subtrees of $\mt{\infty}$.
    Thus, Lemma~\ref{lem:size biased} tells us that, as we make $n$ tends to infinity, the distribution of $(|\mt{\infty}(k)|)_{0\leq k\leq 2r}$ conditionally given that $V_n\in\{\rv^k\}\cup\mt{\infty}(\rv^r)$ converges to a sequence of independent random variables, all distributed according to a $\geom{1-p}$, except for $|\mt{\infty}(r)|$, for which the formula applied with $\pi_\ell=q^{\ell-1}(1-q)$ and $\mu=1/(1-q)$ gives us
    \begin{align*}
        \prob\Big(|\mt{\infty}(r)|=\ell~\Big|~V_n\in\{\rv^k\}\cup\mt{\infty}(\rv^r)\Big)&\longrightarrow\frac{(\ell+1)q^\ell(1-q)}{1/(1-q)}=(\ell+1)q^\ell(1-q)\,;
    \end{align*}
    this is exactly the distribution of a $\sbgeom{1-q}$ as given in~\eqref{eq:all geoms}.
    
    Putting everything together, the distribution of the subtrees $(\mt{\infty}(k))_{0\leq k\leq 2r}$ conditionally given that $V_n\in\{\rv^k\}\cup\mt{\infty}(\rv^r)$ asymptotically converges in distribution to a set of $(2r+1)$ independent Mallows trees of geometric sizes, except for the $r$-th tree whose distribution is size-biased.
    But now, this is exactly the distribution of $(Y_q(k))_{|k|\leq r}$ (by shifting the indices), meaning that
    \begin{align*}
        \lim_{n\rightarrow\infty}\prob\Big(B_r(V_n,\mt{\infty})\equiv H\Big)=\prob\big(B_r(\imt)\equiv H\big)\,,
    \end{align*}
    where we used~\eqref{eq:imt dist} to prove that the right-hand side is as desired.
    This concludes the proof of the proposition.
\end{proof}

With the third proposition proven, we now combine all previous results to prove Theorem~\ref{thm:conv}.

\begin{proof}[Proof of Theorem~\ref{thm:conv}]
    We aim to prove that~\eqref{eq:target} holds, that is
    \begin{align*}
        \frac{1}{n}\sum_{v\in\mt{n}}\mathbbm{1}_{\{B_r(v,\mt{n})\equiv H\}}\asconv\prob\big(B_r(\imt)\equiv H\big)\,.
    \end{align*}
    To do so, we first combine Proposition~\ref{prop:a limit} with Proposition~\ref{prop:correct lim} to obtain that a modified sequence converges as desired:
    \begin{align*}
        \frac{1}{n}\sum_{k=0}^{\lfloor n(1-q)\rfloor}\sum_{v\in\{\rv^k\}\cup\mt{\infty}(k)}\mathbbm{1}_{\{B_r(v,\mt{\infty})\equiv H\}}\asconv\prob\big(B_r(\imt)\equiv H\big)\,.
    \end{align*}
    But now, thanks to Proposition~\ref{prop:new target}, we have that
    \begin{align*}
        \frac{1}{n}\sum_{v\in\mt{n}}\mathbbm{1}_{\{B_r(v,\mt{n})\equiv H\}}-\frac{1}{n}\sum_{k=0}^{\lfloor n(1-q)\rfloor}\sum_{v\in\{\rv^k\}\cup\mt{\infty}(k)}\mathbbm{1}_{\{B_r(v,\mt{\infty})\equiv H\}}\asconv0\,.
    \end{align*}
    The proof of the theorem simply follows from combining the two previous convergence results.
\end{proof}

\section{The other types of convergence}\label{sec:other conv}

After proving the local limit of Mallows trees in Section~\ref{sec:as conv}, we now focus on three other types of convergence and prove Theorem~\ref{thm:rooted}, \ref{thm:ghp}, and~\ref{thm:ssc} respectively in Section~\ref{sec:rooted}, \ref{sec:ghp}, and~\ref{sec:ssc} below.
For the rest of this section, and without loss of generality, we assume that $(\mp{n})_{n\in\bN\cup\{\infty\}}$ are coupled according to the infinite Bernoulli model defined in Section~\ref{sec:finite mallows}.

\subsection{Rooted limit of Mallows trees}\label{sec:rooted}

The aim of this section is to prove Theorem~\ref{thm:rooted}, which corresponds to the limit of Mallows trees according to the rooted topology defined in Section~\ref{sec:other local}.
This result straightforwardly follows from the coupling of the infinite Bernoulli model as defined in Section~\ref{sec:finite mallows}.

\begin{proof}[Proof of Theorem~\ref{thm:rooted}]
    From the definition of $\mp{\infty}$ as the limit of $(\mp{n})_{n\in\bN}$, we know that, for any $r\geq0$, the subtree of $\bst{\mp{n}}$ rooted at $\rv^r\lv$ converges to that of $\bst{\mp{\infty}}$.
    More precisely, the sequence $(\bst{\mp{n}}\cap\rv^r\lv\BST)_{n\in\bN}$ is increasing and satisfies
    \begin{align*}
        \bigcup_{n\in\bN}\Big(\bst{\mp{n}}\cap\rv^r\lv\BST\Big)=\bst{\mp{\infty}}\cap\rv^r\lv\BST\,,
    \end{align*}
    which also implies that
    \begin{align*}
        \bigcup_{n\in\bN}\Big(\bst{\mp{n}}\cap\rv^r\leftT\Big)=\bst{\mp{\infty}}\cap\rv^r\leftT\,,
    \end{align*}
    Since $\bst{\mp{\infty}}\cap\rv^r\leftT$ is finite, there exists $N_0(r)<\infty$ such that, for all $n\geq N_0(r)$, we have
    \begin{align*}
        \bst{\mp{n}}\cap\rv^r\leftT=\bst{\mp{\infty}}\cap\rv^r\leftT\,.
    \end{align*}
    But now, by letting $N(r)=\max\{N_0(k):0\leq k\leq r\}$, we have that $N(r)<\infty$ almost-surely, and for all $n\geq N(r)$
    \begin{align*}
        \bigcup_{k=0}^r\bst{\mp{n}}\cap\rv^k\leftT=\bigcup_{k=0}^r\bst{\mp{\infty}}\cap\rv^k\leftT\,.
    \end{align*}
    In particular, for any $n\geq N(r)$, we also have
    \begin{align*}
        B_r\big(\root,\bst{\mp{n}}\big)=B_r\big(\root,\bst{\mp{\infty}}\big)
    \end{align*}
    and so, for any rooted tree $t$, it follows that
    \begin{align*}
        \lim_{n\rightarrow\infty}\prob\Big(B_r\big(\root,\bst{\mp{n}}\big)\equiv t\Big)=\prob\Big(B_r\big(\root,\bst{\mp{\infty}}\big)\equiv t\Big)\,,
    \end{align*}
    which is exactly the definition of the rooted convergence.
\end{proof}

It is interesting to note that, in the case of the aforementioned coupling, we actually have that
\begin{align*}
    \mathbbm{1}_{\{B_r(\root,\bst{\mp{n}})\equiv t\}}\asconv\mathbbm{1}_{\{B_r(\root,\bst{\mp{\infty}})\equiv t\}}\,,
\end{align*}
which could be seen as an \textit{almost-sure rooted convergence}.
However, the almost-sure rooted convergence simply means that, for $n$ large enough, we almost-surely have that
\begin{align*}
    B_r(\root,\bst{\mp{n}})=B_r(\root,\bst{\mp{\infty}})\,,
\end{align*}
which was the case when considering the coupling of the infinite Bernoulli model but is not necessarily an interesting topology in general.

\subsection{Strong Gromov-Hausdorff-Prokhorov limit of Mallows trees}\label{sec:ghp}

The aim of this section is to prove Theorem~\ref{thm:ghp}, which corresponds to the limit of Mallows trees according to the strong Gromov-Hausdorff-Prokhorov (GHP) topology.
We start by providing a quick definition of this topology (we refer to~\cite[Section~6]{miermont2009tessellations} for further details), which we then use to provide a simpler condition sufficient to prove the limit from Theorem~\ref{thm:ghp}.
We conclude this section with the proof of the theorem.

\medskip

A \textit{weighted metric space} is a triplet $(\metric,d,\mu)$ where $(\metric,d)$ is a compact metric space and $(\metric,\mu)$ is a probability space.
For a pair of weighted metric spaces $\metricspace_1=(\metric_1,d_1,\mu_1)$ and $\metricspace_2=(\metric_2,d_2,\mu_2)$, we define the GHP distance as follows.
First, a \textit{correspondence} between $\metric_1$ and $\metric_2$ is a set $C\subseteq\metric_1\times\metric_2$ such that $\{x_1:\exists(x_1,x_2)\in C\}=\metric_1$ and $\{x_2:\exists(x_1,x_2)\in C\}=\metric_2$ and the \textit{distortion} of $C$ is defined by
\begin{align*}
    \disto(C):=\sup\Big\{\big|d_1(x_1,x'_1)-d_2(x_2,x'_2)\big|:(x_1,x_2),(x'_1,x'_2)\in C\Big\}\,.
\end{align*}
Second, a coupling between $\mu_1$ and $\mu_2$ is a probability measure $M$ on $\metric_1\times\metric_2$ such that the marginal measures of $M$ on $\metric_1$ and $\metric_2$ are respectively $\mu_1$ and $\mu_2$.
Now, the GHP distance on the pair $\metricspace_1$ and $\metricspace_2$ is defined by
\begin{align*}
    \dghp(\metricspace_1,\metricspace_2):=\inf\Big\{\epsilon>0:\exists(C,M),M(C)\geq1-\epsilon,\disto(C)\leq2\epsilon\Big\}\,,
\end{align*}
where the pair $(C,M)$ is composed of a correspondence $C$ between $\metric_1$ and $\metric_2$ and coupling $M$ between $\mu_1$ and $\mu_2$.

While to transform $\dghp$ into a metric on a proper complete separable metric space we need to consider weighted metric spaces up to isomorphism, the definition of $\dghp$ on any pair of weighted metric space is enough for the sake of this study.
Thus, given a sequence of random weighted metric spaces $(\metricspace_n)_{n\in\bN}$ and a metric space $\metricspace$, we say that it converges according to the strong GHP topology, and write it $\metricspace_n\ghpconv\metricspace$ if
\begin{align*}
    \dghp(\metricspace_n,\metricspace)\asconv0\,.
\end{align*}
We call this convergence \textit{strong} in comparison to a weaker one where the convergence would occur in probability (or equivalently in distribution, since the limit is constant).
We now provide a useful condition to check whether $(\metricspace_n)_{n\in\bN}$ converges to the limit in Theorem~\ref{thm:ghp}, that is $\ghplimit=([0,1],d,\mu_{[0,1]})$ where $d(x,y)=|x-y|$ and $\mu_{[0,1]}$ is the uniform measure on $[0,1]$.

\begin{lemma}\label{lem:ghp cond}
    Let $(\metricspace_n)_{n\in\bN}=((\metric_n,d_n,\mu_n))_{n\in\bN}$ be a sequence of metric spaces such that, for all $n\in\bN$, $\metric_n$ is a finite set of size $n$ and $\mu_n$ is the uniform measure on $\metric_n$.
    Assume that, for all $n\in\bN$, there exists an ordering of the elements of $\metric_n=\{x_{i,n}:1\leq i\leq n\}$ such that
    \begin{align}\label{eq:ghp cond}
        \max\left\{\left|d_n(x_{i,n},x_{j,n})-\frac{|i-j|}{n}\right|:1\leq i,j\leq n\right\}\asconv0\,.
    \end{align}
    Then we have the convergence according to the strong GHP topology
    \begin{align*}
        \metricspace_n\ghpconv\ghplimit
    \end{align*}
\end{lemma}

\begin{proof}
    Using the definition of the strong GHP convergence, we see that it suffices to find a sequence $(C_n)_{n\in\bN}$ of correspondences and a sequence $(M_n)_{n\geq1}$ of matching such that, $C_n$ is a correspondence between $\metric_n$ and $[0,1]$ satisfying
    \begin{align*}
        \disto(C_n)\asconv0
    \end{align*}
    and $M_n$ is a matching between $\mu_n$ and $\mu_{[0,1]}$ satisfying
    \begin{align*}
        M_n(C_n)\asconv1\,.
    \end{align*}
    To do so we rely on the ordering of $\metric_n$ provided as assumption of the lemma.

    \medskip

    Consider first the correspondence defined by
    \begin{align*}
        C_n=\left\{\left(x_{i,n},\left[\frac{i-1}{n},\frac{i}{n}\right]\right):1\leq i\leq n\right\}\,.
    \end{align*}
    It is indeed a correspondence between $\metric_n$ and $[0,1]$ since $\metric_n=\{x_{i,n}:1\leq i\leq n\}$ and $[0,1]=\bigcup_{i=1}^n[(i-1)/n,i/n]$.
    Now, $C_n$ is a union of $n$ disjoint segments of $\metric_n\times[0,1]$, so choose $M_n$ to be the uniform probability measure on $C_n$.
    By the definition of $C_n$, $M_n$ is indeed a matching between $\mu_n$ and $\mu_{[0,1]}$ since $M_n(\{x_{i,n}\}\times[0,1])=M_n(\{x_{i,n}\times[(i-1)/n,i/n]\})=\frac{1}{n}$ and, for any Lebesgue measurable set $A$ of $[0,1]$,
    \begin{align*}
        M_n(\metric_n\times A)&=\sum_{i=1}^nM_n(\{x_{i,n}\}\times(A\cap[(i-1)/n,i/n]))\\
        &=\sum_{i=1}^n\mu_{[0,1]}\big(A\cap[(i-1)/n,i/n]\big)\\
        &=\mu_{[0,1]}(A)\,,
    \end{align*}
    where the last equality uses that the union of the singletons $\{1/n,2/n,\ldots,(n-1)/n\}$ has $\mu_{[0,1]}$-measure $0$.
    This proves that $(C_n)_{n\in\bN}$ and $(M_n)_{n\in\bN}$ are correspondences and matching as desired.
    Moreover, by definition $M_n(C_n)=1$, so it only remains to prove that $\disto(C_n)\asconv0$.
    This directly follows from the assumption of the lemma since
    \begin{align*}
        \disto(C_n)&=\sup\left\{\left|d_n(x_{i,n},x_{j,n})-\frac{|x-y|}{n}\right|:1\leq i,j\leq n,i-1\leq x\leq i,j-1\leq y\leq j\right\}\\
        &\leq\frac{2}{n}+\max\left\{\left|d_n(x_{i,n},x_{j,n})-\frac{|i-j|}{n}\right|:1\leq i,j\leq n\right\}\,.
    \end{align*}
    The right-hand side of this equation converges almost-surely to $0$ and so $\disto(C_n)\asconv0$, proving that $\dghp(\metricspace_n,\ghplimit)\asconv0$ as desired.
\end{proof}

We now use the condition stated in this lemma to prove the convergence result from Theorem~\ref{thm:ghp}.

\begin{proof}[Proof of Theorem~\ref{thm:ghp}]
    In order to prove that the theorem holds, we simply show that $(\bst{\mp{n}},((1-q)n)^{-1}\cdot d_n,\mu_n)$ satisfies the conditions of Lemma~\ref{lem:ghp cond}.
    First of all, note that $\bst{\mp{n}}$ is indeed a set of size $n$ and that $\mu_n$ is, by definition, the uniform measure on this set.
    It thus only remains to prove that there exists an ordering of the elements of $\bst{\mp{n}}$ satisfying~\eqref{eq:ghp cond}.

    Fix $n\in\bN$ for now and organize $\bst{\mp{n}}$ into its sequence of left subtrees $T_{0,n},\ldots,T_{r_n,n}$.
    More precisely, we have $r_n=\reco{n}$ as defined in~\eqref{eq:records} and
    \begin{align*}
        T_{k,n}=\bst{\mp{n}}\cap\rv^k\leftT\subseteq\{\rv^k\}\cup\Big(\bst{\mp{\infty}}\cap\rv^k\lv\BST\Big)\,.
    \end{align*}
    Now, for any $0\leq k\leq r$, let $x^k=(x^{i,k}:1\leq i\leq|T_{k,n}|)$ be an arbitrary ordering of the set $T_{k,n}$ and define $(x_{i,n})_{1\leq i\leq n}$ as the concatenation of the $(x^k)_{0\leq k\leq r_n}$, that is 
    \begin{align*}
        (x_{1,n},\ldots,x_{n,n})=\left(x^{0,1},\ldots,x^{0,|T_{0,n}|},x^{1,1},\ldots,x^{r_n,|T_{r_n,n}|}\right)\,.
    \end{align*}
    In words, the sequence $(x_{i,n})_{1\leq i\leq n}$ goes in order through the left subtrees of $\bst{\mp{n}}$ and explore each such left subtree in an arbitrary order.
    We now prove that this sequence satisfies~\eqref{eq:ghp cond}.

    \medskip

    Consider $1\leq i,j\leq n$.
    Note that $x_{i,n}$ belongs to the left subtree $T_{k,n}$ where $k=k_n(i)$ is the minimal integer satisfying
    \begin{align*}
        i\leq\big|T_{0,n}\big|+\ldots+\big|T_{k,n}\big|\,.
    \end{align*}
    Similarly, $x_{j,n}$ belongs to the left subtree $T_{\ell,n}$ where $\ell=k_n(j)$ is minimal so that
    \begin{align*}
        j\leq\big|T_{0,n}\big|+\ldots+\big|T_{\ell,n}\big|\,.
    \end{align*}
    Since the minimal distance between two vertices of $T_{k,n}$ and $T_{\ell,n}$ is attained when considering $\rv^k$ and $\rv^\ell$, whose distance is $|k-\ell|$, we have that
    \begin{align*}
        |k-\ell|\leq d_n(x_{i,n},x_{j,n})\leq|k-\ell|+|T_{k,n}|+|T_{\ell,n}|\,.
    \end{align*}
    By taking the maximum over all the left subtree sizes, we obtain that
    \begin{align}\label{eq:new cond}
        \max\left\{\left|\frac{d_n(x_{i,n},x_{j,n})}{(1-q)n}-\frac{|k_n(i)-k_n(j)|}{(1-q)n}\right|:1\leq i,j\leq n\right\}\leq\frac{2}{(1-q)n}\max\Big\{|T_{k,n}|:0\leq k\leq r_n\Big\}\,.
    \end{align}
    We now show that the right-hand side of this equation converges to $0$ almost-surely.

    Using the definition of $T_{k,n}$ and the coupling of the infinite Bernoulli model, we have that
    \begin{align*}
        T_{k,n}=\bst{\mp{n}}\cap\rv^k\leftT\subseteq\{\rv^k\}\cup\Big(\bst{\mp{\infty}}\cap\rv^k\lv\BST\Big)\,.
    \end{align*}
    We also recall from~\cite[Lemma~1.10]{addario2021height} that the family of random variables $(|\bst{\mp{\infty}}\cap\rv^k\lv\BST|)_{k\geq0}$ is a sequence of independent $\geom{1-q}$.
    Moreover, using that $\reco{n}\leq n$, it follows that
    \begin{align*}
        \max\Big\{|T_{k,r}|:0\leq k\leq r_n\Big\}\leq1+\max\Big\{\big|\bst{\mp{\infty}}\cap\rv^k\lv\BST\big|:0\leq k\leq n\Big\}\,.
    \end{align*}
    Thus, for any $\epsilon>0$, we have
    \begin{align*}
        &\prob\left(\frac{2}{(1-q)n}\max\Big\{|T_{k,r}|:0\leq k\leq r_n\Big\}>\epsilon\right)\\
        &\pskip\leq\prob\left(\max\Big\{\big|\bst{\mp{\infty}}\cap\rv^k\lv\BST\big|:0\leq k\leq n\Big\}\geq\frac{\epsilon(1-q)n}{2}-1\right)\\
        &\pskip=1-\left(1-q^{\left\lceil\frac{\epsilon(1-q)n}{2}-1\right\rceil}\right)^{n+1}\\
        &\pskip=\big(1+o(1)\big)nq^\frac{\epsilon(1-q)n}{2}\,,
    \end{align*}
    and the almost-sure convergences follows from applying the Borel-Cantelli lemma.
    This proves that the right-hand side of~\eqref{eq:new cond} converges alsmot-surely to $0$ and so
    \begin{align}\label{eq:first conv}
        \max\left\{\left|\frac{d_n(x_{i,n},x_{j,n})}{(1-q)n}-\frac{|k_n(i)-k_n(j)|}{(1-q)n}\right|:1\leq i,j\leq n\right\}\asconv0
    \end{align}
    This is almost exactly~\eqref{eq:ghp cond} and we now prove that $|k_n(i)-\ell_n(j)|/(1-q)$ can be replaced with $|i-j|$ by proving that
    \begin{align}\label{eq:second conv}
        \max\left\{\left|\frac{i}{n}-\frac{k_n(i)}{(1-q)n}\right|:1\leq i\leq n\right\}\asconv0\,.
    \end{align}

    \medskip

    Recall the definition of $k=k_n(i)$ as the unique integer such that
    \begin{align*}
        \big|T_{0,n}\big|+\ldots+\big|T_{k-1,n}\big|<i\leq\big|T_{0,n}\big|+\ldots+\big|T_{k,n}\big|\,.
    \end{align*}
    First, using that all the left subtrees have size at least $1$, we see that $k\leq i$.
    Let now $(a_n)_{n\in\bN}$ be a sequence of integers such that $a_n/n\rightarrow0$ but $a_n/\log n\rightarrow\infty$ (for example $a_n=\lfloor\log^2n\rfloor$ or $a_n=\lfloor\sqrt{n}\rfloor$).
    Then, the error in~\eqref{eq:second conv} by removing the first $a_n$ term is negligible, as we observed that for $i\leq a_n$ we had $k\leq a_n$, implying
    \begin{align}\label{eq:clipping an}
        \left|\max\left\{\left|\frac{i}{n}-\frac{k_n(i)}{(1-q)n}\right|:1\leq i\leq n\right\}-\max\left\{\left|\frac{i}{n}-\frac{k_n(i)}{(1-q)n}\right|:a_n\leq i\leq n\right\}\right|\leq\frac{a_n}{n}+\frac{a_n}{(1-q)n}\,.
    \end{align}
    Thus, it suffices to prove that the convergence in~\eqref{eq:second conv} is true for $a_n\leq i\leq n$.

    Recall the definition of $\free{n}$ from~\eqref{eq:spaces}, to see that, for any $k\geq0$
    \begin{align*}
        \big|T_{0,n}\big|+\ldots+\big|T_{k-1,n}\big|\geq k+\sum_{\ell=0}^{k-1}\big|\bst{\mp{\infty}}\cap\rv^\ell\lv\BST\big|-\free{n}\,.
    \end{align*}
    Combining this inequality with the fact that $\bst{\mp{n}}$ is a subtree of $\bst{\mp{\infty}}$ and the definition of $k(i)$, we obtain
    \begin{align}\label{eq:bound k and i}
        k(i)+\sum_{\ell=0}^{k(i)-1}\big|\bst{\mp{\infty}}\cap\rv^\ell\lv\BST\big|-\free{n}<i\leq k(i)+1+\sum_{\ell=0}^{k(i)}\big|\bst{\mp{\infty}}\cap\rv^\ell\lv\BST\big|\,.
    \end{align}
    Now, using the fact that $(|\bst{\mp{\infty}}\cap\rv^k\lv\BST|)_{k\geq0}$ is a family of independent $\geom{1-q}$ random variables along with the strong law of large number implies that, as $k$ goes to infinity, we have
    \begin{align*}
        \frac{1}{k}\sum_{\ell=0}^k\big|\bst{\mp{\infty}}\cap\rv^\ell\lv\BST\big|\asconv\frac{q}{1-q}\,.
    \end{align*}
    Moreover, using the upper bound of~\eqref{eq:bound k and i} with $i=a_n$, we can verify that
    \begin{align*}
        \frac{a_n}{k_n(a_n)}\leq1+\frac{1}{k_n(a_n)}+\frac{1}{k_n(a_n)}\sum_{\ell=0}^{k_n(a_n)}\big|\bst{\mp{\infty}}\cap\rv^\ell\lv\BST\big|\,,
    \end{align*}
    and so $k_n(a_n)\asconv\infty$.
    Since $i\mapsto k_n(i)$ is increasing this implies that, uniformly over all $a_n\leq i\leq n$, we have
    \begin{align*}
        \frac{1}{k_n(i)}\sum_{\ell=0}^{k_n(i)}\big|\bst{\mp{\infty}}\cap\rv^\ell\lv\BST\big|\asconv\frac{q}{1-q}\,.
    \end{align*}
    Combining this result with the upper bound of~\eqref{eq:bound k and i}, it follows that, for $n$ large enough, for any $a_n\leq i\leq n$, we almost-surely have that
    \begin{align*}
        \frac{i}{k_n(i)}&\leq1+\frac{1}{k_n(i)}+\frac{1}{k_n(i)}\sum_{\ell=0}^{k_n(i)}\big|\bst{\mp{\infty}}\cap\rv^\ell\lv\BST\big|\\
        &\leq1+1+\frac{2q}{1-q}\\
        &=\frac{2}{1-q}\,,
    \end{align*}
    and so $k_n(i)\geq(1-q)i/2\geq(1-q)a_n/2$.

    Recall the bounds from~\eqref{eq:bound k and i} and use the previous convergence results related to $k_n(i)$ so see that
    \begin{align*}
        \left|\frac{i}{k_n(i)}-\frac{1}{1-q}\right|&\leq\left|\frac{1}{k_n(i)}\sum_{\ell=0}^{k_n(i)}\big|\bst{\mp{\infty}}\cap\rv^\ell\lv\BST\big|-\frac{q}{1-q}\right|\\
        &\pskip+\frac{1}{k_n(i)}+\frac{\big|\bst{\mp{\infty}}\cap\rv^{k_n(i)}\lv\BST\big|+\free{n}}{k_n(i)}\,.
    \end{align*}
    Thus, thanks to $k_n(i)\geq k_n(a_n)\geq(1-q)a_n/2$ almost surely for $n$ large enough, the fact that the left subtrees are almost surely finite, and the result from Lemma~\ref{lem:spaces}, satisfied here since $a_n/\log n\rightarrow\infty$, we eventually obtain that
    \begin{align*}
        \frac{i}{k_n(i)}\asconv\frac{1}{1-q}\,;
    \end{align*}
    moreover, this convergence is uniform over $a_n\leq i\leq n$, meaning that
    \begin{align*}
        \max\left\{\left|\frac{i}{k_n(i)}-\frac{1}{1-q}\right|:a_n\leq i\leq n\right\}\asconv0\,.
    \end{align*}
    From this convergence, we finally obtain that
    \begin{align}\label{eq:third conv}
        \max\left\{\left|\frac{i}{n}-\frac{k_n(i)}{(1-q)n}\right|:a_n\leq i\leq n\right\}&\leq\max\left\{\frac{k_n(i)}{n}\cdot\left|\frac{i}{k_n(i)}-\frac{1}{1-q}\right|:a_n\leq i\leq n\right\}\\
        &\leq\max\left\{\left|\frac{i}{k_n(i)}-\frac{1}{1-q}\right|:a_n\leq i\leq n\right\}\asconv0\,,\notag
    \end{align}
    where we use that $k_n(i)\leq i\leq n$ for the second inequality.
    We now have all the results to prove the theorem.

    \medskip

    First of all, thanks to~\eqref{eq:third conv}, we know that
    \begin{align*}
        \max\left\{\left|\frac{i}{n}-\frac{k_n(i)}{(1-q)n}\right|:a_n\leq i\leq n\right\}\asconv0\,.
    \end{align*}
    We further recall that $a_n/n\rightarrow\infty$ which, along with~\eqref{eq:clipping an}
    \begin{align*}
        \left|\max\left\{\left|\frac{i}{n}-\frac{k_n(i)}{(1-q)n}\right|:1\leq i\leq n\right\}-\max\left\{\left|\frac{i}{n}-\frac{k_n(i)}{(1-q)n}\right|:a_n\leq i\leq n\right\}\right|\leq\frac{a_n}{n}+\frac{a_n}{(1-q)n}\,,
    \end{align*}
    implies that~\eqref{eq:second conv} holds:
    \begin{align*}
        \max\left\{\left|\frac{i}{n}-\frac{k_n(i)}{(1-q)n}\right|:1\leq i\leq n\right\}\asconv0\,.
    \end{align*}
    Now, we also proved~\eqref{eq:first conv}, stating that
    \begin{align*}
        \max\left\{\left|\frac{d_n(x_{i,n},x_{j,n})}{(1-q)n}-\frac{|k_n(i)-k_n(j)|}{(1-q)n}\right|:1\leq i,j\leq n\right\}\asconv0\,.
    \end{align*}
    Moreover, for any $1\leq i,j\leq n$, we have
    \begin{align*}
        \left|\frac{d_n(x_{i,n},x_{j,n})}{(1-q)n}-\frac{|i-j|}{n}\right|\leq\left|\frac{d_n(x_{i,n},x_{j,n})}{(1-q)n}-\frac{|k_n(i)-k_n(j)|}{(1-q)n}\right|+\left|\frac{k_n(i)}{(1-q)n}-\frac{i}{n}\right|+\left|\frac{k_n(j)}{(1-q)n}-\frac{j}{n}\right|\,.
    \end{align*}
    Thus, combining the three previous equations, we see that
    \begin{align*}
        \max\left\{\left|\frac{d_n(x_{i,n},x_{j,n})}{(1-q)n}-\frac{|i-j|}{n}\right|:1\leq i,j\leq n\right\}\asconv0\,,
    \end{align*}
    which is exactly the condition~\eqref{eq:ghp cond} from Lemma~\ref{lem:ghp cond}.
    This finally prove that $(\bst{\mp{n}},((1-q)n)^{-1}\cdot d_n,\mu_n)$ satisfies all the conditions of Lemma~\ref{lem:ghp cond} and thus implies that Theorem~\ref{thm:ghp} holds.
\end{proof}

\subsection{Subtree size limit of Mallows trees}\label{sec:ssc}

The aim of this section is to prove Theorem~\ref{thm:ssc}, which corresponds to the limit of Mallows trees according to the subtree size topology, first introduced in~\cite{grubel2023note}.
We start by providing some background on this topology, before proving the desired limit.

First, we define the set of functions on $\BST$ with mass $1$ at the root, and whose mass split at every node:
\begin{align*}
    \Psi:=\Big\{\psi:\BST\rightarrow[0,1]~\Big|~\psi(\root)=1~\&~\forall v\in\BST,\psi(v)=\psi(v\lv)+\psi(v\rv)\Big\}
\end{align*}
Now, given a finite binary search tree $T$, we can consider the corresponding function $\ss_T$ defined by
\begin{align*}
    \ss_T(v):=\frac{|T\cap v\BST|}{|T|}\,.
\end{align*}
Note that $\ss_T$ does not belong to $\Psi$ since $\tau_T(\lv)+\tau_T(\rv)=1-\frac{1}{|T|}$, but its limit will, provided that $|T|\rightarrow\infty$.
This is the essence of the subtree size convergence.

Consider a sequence of random binary search trees $(T_n)_{n\in\bN}$ such that $|T_n|\rightarrow\infty$ and let $\psi\in\Psi$.
We say that the $(T_n)_{n\in\bN}$ \textit{converge with respect to the subtree sizes} to $\psi$, and write it $T_n\subtreeconv\psi$ if $(\tau_{T_n})_{n\in\bN}$ converges in distribution towards $\psi$.
In other words, $(T_n)_{n\in\bN}$ converge with respect to the subtree sizes to $\psi$ if and only if, for all $r\geq1$ and $v_1,\ldots,v_r\in\BST$, we have
\begin{align*}
    \big(\tau_{T_n}(v_i)\big)_{1\leq i\leq r}\longrightarrow\big(\psi(v_i)\big)_{1\leq i\leq r}\,,
\end{align*}
where the convergence occurs in distribution.

Naturally, following the observations from Section~\ref{sec:local} and the remark at the end of Section~\ref{sec:rooted}, we could consider this type of convergence as the \textit{weak} convergence with respect to the subtree sizes, and could also define the \textit{almost-sure convergence with respect to the subtree sizes} (or the \textit{convergence in probability with respect to the subtree sizes}) by replacing the previous convergence in distribution with an almost-sure convergence (or a convergence in probability).
Actually, in the case of the coupling of the infinite Bernoulli model defined in Section~\ref{sec:finite mallows}, we will show that the previous convergence occurs almost-surely, thus corresponding to an almost-sure convergence with respect to the subtree sizes.
This also naturally proves the regular subtree size convergence with respect to the weak topology.

\begin{proof}[Proof of Theorem~\ref{thm:ssc}]
    We want to show that $(\bst{\mp{n}})_{n\in\bN}$ converges to $\psi_R$ defined in~\eqref{eq:def psiR}, that is, for all $v_1,\ldots,v_r$, we have
    \begin{align*}
        \big(\ssm{n}(v_i)\big)_{1\leq i\leq r}\longrightarrow\big(\psi_r(v_i)\big)_{1\leq i\leq r}\,,
    \end{align*}
    where $\ssm{n}=\tau_{\bst{\mp{n}}}$.
    Fix now $r\geq1$ and consider the sequence $(v_0,\ldots,v_R)$ defined by $R=2^{r+1}-2$, $v_i=\rv^i$ for $0\leq i\leq r$, and the set $\{v_i:0\leq i\leq R\}$ corresponds to all the nodes of $\BST$ at depth at most $R$.
    Note that the ordering of $v_i$ for $i>r$ will not matter.
    With this sequence defined, we will now show that
    \begin{align*}
        \big(\ssm{n}(v_i)\big)_{0\leq i\leq R}\asconv\big(\psi_R(v_i)\big)_{0\leq i\leq R}=\big(\underset{r+1}{\underbrace{1,\ldots,1}},0,\ldots,0\big)\,,
    \end{align*}
    which will then imply the desired convergence with respect to the subtree sizes.
    Since this sequence is finite, the previous convergence result is actually equivalent to proving that
    \begin{align*}
        \ssm{n}(v_i)\asconv\left\{\begin{array}{ll}
            1 & \textrm{if $0\leq i\leq r$} \\
            0 & \textrm{otherwise}\,.
        \end{array}\right.
    \end{align*}
    We now focus on proving this last statement.

    First, for any $0\leq i\leq r$, we have $v_i=\rv^i$ and so
    \begin{align*}
        \ssm{n}(v_i)&=\frac{1}{n}\Big|\bst{\mp{n}}\cap\rv^i\BST\Big|=\frac{1}{n}\left[n-\sum_{j=0}^{i-1}\Big(1+\big|\bst{\mp{n}}\cap\rv^j\lv\BST\big|\Big)\right]\,.
    \end{align*}
    Using the coupling of the infinite Bernoulli model, we see that $\bst{\mp{n}}$ is a subtree of $\bst{\mp{\infty}}$, and so we can bound the previous term as follows
    \begin{align*}
        \ssm{n}(v_i)&\geq1-\frac{1}{n}\sum_{j=0}^{i-1}\Big(1+\big|\bst{\mp{\infty}}\cap\rv^j\lv\BST\big|\Big)\geq1-\frac{1}{n}\sum_{j=0}^r\Big(1+\big|\bst{\mp{\infty}}\cap\rv^j\lv\BST\big|\Big)\,.
    \end{align*}
    Moreover, we know that the sequence $(|\bst{\mp{\infty}}\cap\rv^k\lv\BST|)_{k\geq0}$ is a family of independent $\geom{1-q}$ random variables~\cite[Lemma~1.10]{addario2021height}, so the term
    \begin{align*}
        E=\sum_{j=0}^r\Big(1+\big|\bst{\mp{\infty}}\cap\rv^j\lv\BST\big|\Big)
    \end{align*}
    satisfies
    \begin{align*}
        \expec[E]=(r+1)\left(1+\frac{q}{1-q}\right)=\frac{1+r}{1-q}<\infty\,,
    \end{align*}
    and is thus almost-surely finite.
    This implies that $E/n\asconv0$ and so, using that $\ssm{n}(v_i)\leq1$, we obtain that
    \begin{align*}
        \ssm{n}(v_i)\asconv1
    \end{align*}
    for any $0\leq i\leq r$.

    Now, for any $i>r$, there exists some $0\leq j\leq r$ such that $v_i\in\rv^j\lv\BST$, and so
    \begin{align*}
        \ssm{n}(v_i)&\leq\frac{1}{n}\Big|\bst{\mp{n}}\cap\rv^j\lv\BST\Big|\leq\frac{E}{n}\,.
    \end{align*}
    Thus, by using that $E$ is almost-surely finite, we obtain that
    \begin{align*}
        \ssm{n}(v_i)\asconv0
    \end{align*}
    for any $i>r$, thus proving the desired convergence result. 
\end{proof}

\section{Discussions}

In this section we briefly discuss how we can extend the results of this article to the case $q>1$ and mentiond related open problems.

\subsection{Ledwood trees}\label{sec:ledwood}

All the results from this work are applied to Mallows trees with $q\in[0,1)$.
As mentioned in Section~\ref{sec:main}, this choice makes sense since the case $q=1$ behaves very differently and the case $q>1$ can be obtained by symmetry.
We now explain how to extend the results to the case $q>1$ and will further discuss the case $q=1$ in Section~~\ref{sec:follow-up}.

A known result regarding Mallows trees is that, if $T_{n,q}$ is a Mallows trees of size $n$ with parameter $q\in(0,\infty)$, then $\overline{T}_{n,q}$ obtained by swapping all left and right subtrees is also a Mallows tree of size $n$ but with parameter $1/q$.
In a similar manner, of $\imp$ is a two-sided infinite Mallows permutation with parameter $q$, then $-\imp$ is a two-sided infinite Mallows permutation with parameter $1/q$.
Finally, we can define ledwood trees as subtrees of $\overline{\RIBST}$ where $\overline{\RIBST}$ is defined in the same way as $\RIBST$ except that $\uv$ is now considered the inverse of $\lv$ instead of that of $\rv$.
Under this new definition, for any redwood tree $Y$, we can define the ledwood tree $\overline{Y}$ where we transform each $\lv$ into $\rv$ and vice-versa, but do not change $\uv$.
Similarly, we can extend $\ribst{\cdot}$ to $\overline{Y}\langle\cdot\rangle$ either by adapting the definition from~\eqref{eq:redwood}, or by letting $\overline{Y}\langle\mathcal{X}\rangle=\overline{\ribst{-\mathcal{X}}}$.

With all these notations, if $T_{n,q}$ is a Mallows tree of size $n$ with parameter $q>1$, then $\overline{T}_{n,q}$ is a Mallows tree of size $n$ with parameter $1/q<1$ and so, by Theorem~\ref{thm:conv}, we have
\begin{align*}
    \overline{T}_{n,q}\asconv\mathcal{Y}_{1/q}=\left(o,\ribst{\mathcal{X}_{1/q}}\right)\in\RIBST\,.
\end{align*}
But now, $\overline{\ribst{\mathcal{X}_{1/q}}}\overset{d}{=}\overline{\ribst{-\imp}}=\overline{Y}\left\langle\imp\right\rangle$, leading to
\begin{align*}
    T_{n,q}\asconv\left(o,\overline{Y}\left\langle\imp\right\rangle\right)\,,
\end{align*}
which is exactly the extension of Theorem~\ref{thm:conv} to the case $q>1$ simply obtained by replacing the redwood tree of $\imp$ (when $q<1$) by its ledwood tree (when $q>1$).

\medskip

All other results extend in a similar manner, and we only briefly describe the changes below.
\begin{itemize}
    \item
    Theorem~\ref{thm:imt}:
    swap the role of $\lv$ and $\rv$, keeping $\uv$ as it is but with the convention that $\uv=\lv^{-1}$ so $\lv^k=\rv^{|k|}$ for $k<0$ and $\rv^k$ is only defined for $k\geq0$.
    \item
    Theorem~\ref{thm:rooted}:
    here there is a small problem since one-sided infinite Mallows permutations are not define when $q>1$;
    thus we can either show that the limit when $q>1$ is just the symmetric of the limit when $q<1$, or we can define \textit{infinite Mallows functions} as bijective functions from $\bN$ to $-\bN=\{\ldots,-2,-1\}$ whose weight is proportional to $q^{\inv(\sigma)}$, having in mind that this is only properly defined when $q>1$.
    \item
    Theorem~\ref{thm:ghp}:
    the limit in this case actually does not change, we just need to replace the distance by $((1-1/q)n)^{-1}\cdot d_n$ instead of $((1-q)n)^{-1}\cdot d_n$.
    \item
    Theorem~\ref{thm:ssc}:
    we replace the limit by its symmetric, that is $\psi_L=\mathbbm{1}_{v\in\{\lv^k:k\geq0\}}$.
\end{itemize}

\subsection{Open problems}\label{sec:follow-up}

From the previous results and observations, we see that only the case $q=1$ remains to be studied in order to obtain a full characterization of the limit of Mallows trees.
However, this case behaves very differently than when $q\neq1$ and we thus believe that such result would be harder to prove.
Moreover, since we expect a sharp change of behaviour in the limit of Mallows trees between the cases $q\neq1$ and $q=1$, we would also likely observe various regimes if we let $q=q_n$ depend on $n$ (as it was observed in the case of the height of the trees for example~\cite{addario2021height}).
Characterizing these limits and how they depend on the relation between $q$ and $n$ are likely challenging questions.

A possibly easier follow-up project would be to extend these results to the case of random regenerative permutations~\cite{pitman2019regenerative}.
These permutations, which are known to generalize Mallows permutations, show strong regenerative properties, which happen to be one of the most important property in the study of the limit of Mallows trees.
While this study would naturally build off the tools developed here, it is not clear whether the limit could also be characterized in two different ways: as a sequence of trees attached to a straight line and as a redwood tree built off a two-sided infinite permutation.
For this reason we decided to focus on the case of Mallows permutations here and keep the case of random regenerative permutations for a possible later work.

\section*{Acknowledgement}

This project has received funding from the European Union's Horizon 2020 research and innovation programme under the Marie Sk\l{}odowska-Curie grant agreement Grant Agreement No.~101034253~\includegraphics[height=\EUflag]{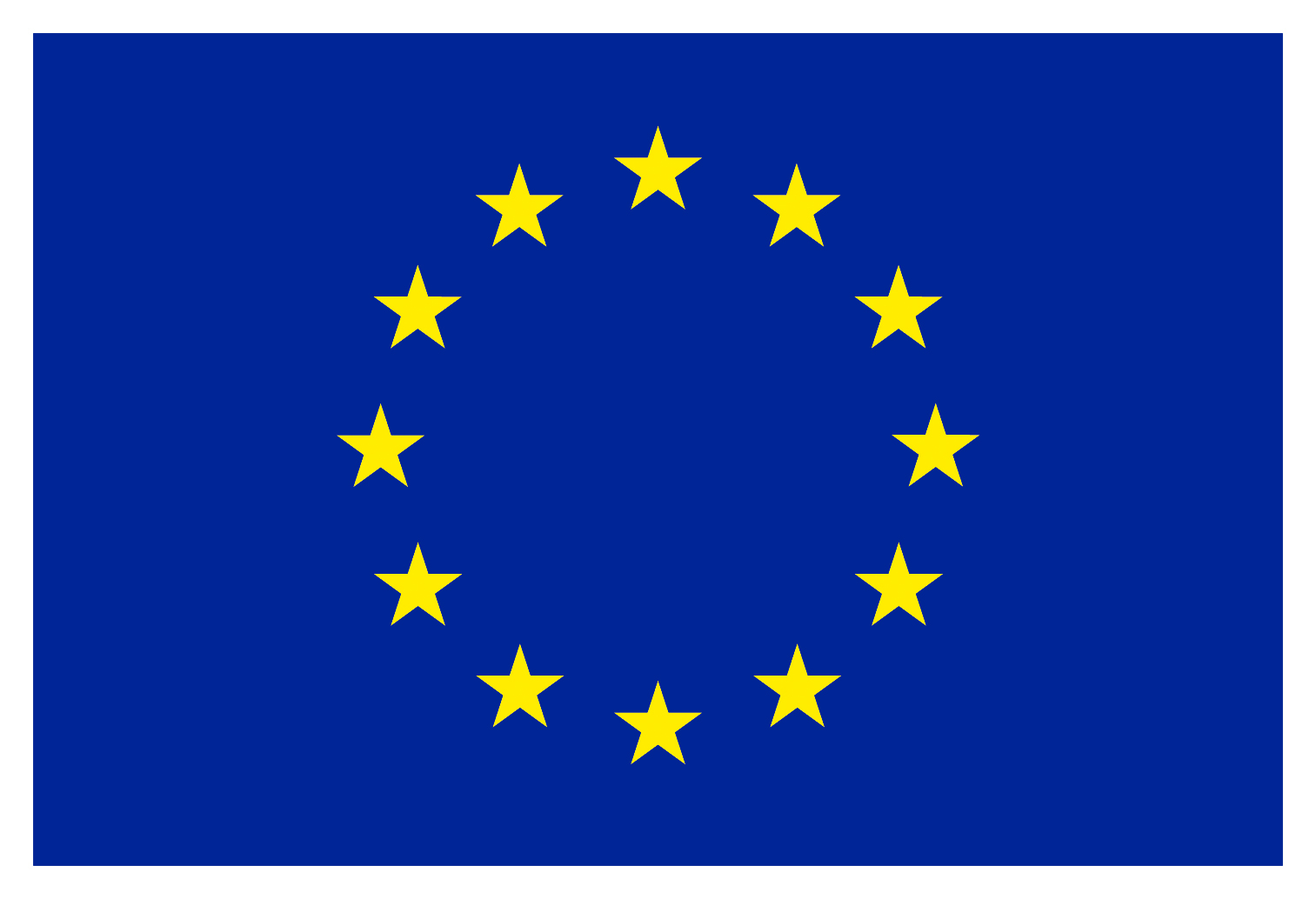}.

\bibliographystyle{alpha}
\bibliography{main}

\end{document}